\documentclass[12pt]{amsart}
\usepackage{a4wide,enumerate,color}
\usepackage{graphicx,epsfig,subfigure}
\usepackage{psfrag}

\allowdisplaybreaks

\let\pa\partial  
\let\na\nabla  
\let\eps\varepsilon  
\newcommand{\N}{{\mathbb N}}  
\newcommand{\R}{{\mathbb R}} 
\newcommand{\diver}{\operatorname{div}}

\newcommand{\m}{\mbox{\rm m}}

\newtheorem{theorem}{Theorem}   
\newtheorem{lemma}[theorem]{Lemma}

\def\txtd{{\textnormal{d}}}
\def\txte{{\textnormal{e}}}

\def\txtD{{\textnormal{D}}}

\newcommand{\be}{\begin{equation}}
\newcommand{\ee}{\end{equation}}
\newcommand{\bea}{\begin{eqnarray}}
\newcommand{\eea}{\end{eqnarray}}
\newcommand{\beann}{\begin{eqnarray*}}
\newcommand{\eeann}{\end{eqnarray*}}
\newcommand{\benn}{\begin{equation*}}
\newcommand{\eenn}{\end{equation*}}

\def\ra{\rightarrow}
\def\I{\infty}

\newcommand{\cA}{{\mathcal A}}  
\newcommand{\cB}{{\mathcal B}}  
\newcommand{\cD}{{\mathcal D}}  
\newcommand{\cF}{{\mathcal F}}  
\newcommand{\cN}{{\mathcal N}}  
\newcommand{\cR}{{\mathcal R}}  
\newcommand{\cX}{{\mathcal X}}  
\newcommand{\cY}{{\mathcal Y}}  

\def\txtd{{\textnormal{d}}}
\def\txte{{\textnormal{e}}}

\def\txtD{{\textnormal{D}}}

\begin{document}  

\title[Cross-diffusion herding]{A meeting point of entropy and bifurcations 
in cross-diffusion herding}

\author{Ansgar J\"ungel}
\address{Institute for Analysis and Scientific Computing, Vienna University of  
	Technology, Wiedner Hauptstra\ss e 8--10, 1040 Wien, Austria}
\email{juengel@tuwien.ac.at} 

\author{Christian Kuehn}
\address{Institute for Analysis and Scientific Computing, Vienna University of  
	Technology, Wiedner Hauptstra\ss e 8--10, 1040 Wien, Austria}
\email{ck274@cornell.edu}

\author{Lara Trussardi}
\address{Institute for Analysis and Scientific Computing, Vienna University of  
	Technology, Wiedner Hauptstra\ss e 8--10, 1040 Wien, Austria}
\email{lara.trussardi@tuwien.ac.at}

\thanks{AJ and LT acknowledge partial support from   
the European Union in the FP7-PEOPLE-2012-ITN Program under 
Grant Agreement Number 304617,
the Austrian Science Fund (FWF), grants P22108, P24304, W1245, and
the Austrian-French Program of the Austrian Exchange Service (\"OAD).
CK acknowledges partial support by an APART fellowship of the Austrian
Academy of Sciences (\"OAW) and by a Marie-Curie International Reintegration
Grant by the EU/REA (IRG 271086). Furthermore, we would like to thank two
anonymous referees for very helpful comments and suggestions that led to 
several improvements.} 

\begin{abstract}
A cross-diffusion system modeling the information herding of individuals is analyzed
in a bounded domain with no-flux boundary conditions. The variables are the species' 
density and an influence function which modifies the information state of the individuals. 
The cross-diffusion term may stabilize or destabilize the system. Furthermore, 
it allows for a formal gradient-flow or entropy structure. Exploiting this structure, 
the global-in-time existence of weak solutions and the exponential decay to the 
constant steady state is proved in certain parameter regimes. This approach
does not extend to all parameters. We investigate local bifurcations from
homogeneous steady states analytically to determine whether this defines the validity
boundary. This analysis shows that generically there is a gap in the parameter 
regime between the entropy approach validity and the first local bifurcation. Next, 
we use numerical continuation methods to track the bifurcating non-homogeneous 
steady states globally and to determine non-trivial stationary solutions related to
herding behaviour. In summary, we find that the main boundaries in the parameter regime 
are given by the first local bifurcation point, the degeneracy of the diffusion matrix 
and a certain entropy decay validity condition. We study several parameter limits 
analytically as well as numerically, with a focus on the role of changing a linear 
damping parameter as well as a parameter controlling the cross-diffusion. We suggest 
that our paradigm of comparing bifurcation-generated obstructions to the parameter 
validity of global-functional methods could also be of relevance for many other models 
beyond the one studied here.
\end{abstract}

$\quad$

\keywords{Information herding, entropy method, global existence of solutions,
large-time dynamics of solutions, relative entropy, Crandall-Rabinowitz, numerical 
continuation, bifurcation.}
 
\subjclass[2000]{35K57, 35K20, 35B40, 35Q91}  

\maketitle

\section{Introduction}
\label{sec.intro}

In this paper we study the following cross-diffusion system:
\begin{align}
  \pa_t u_1 &= \diver(\na u_1 - g(u_1)\na u_2), \label{1.eq1} \\
  \pa_t u_2 &= \diver(\delta\na u_1 + \kappa \na u_2) + f(u_1)-\alpha u_2,
	\label{1.eq2}
\end{align}
where $u_1=u_1(t,x)$, $u_2=u_2(t,x)$ for $(t,x)\in[0,T)\times \Omega$, $T>0$ is the final
time, $\Omega\subset\R^d$ ($d\ge 1$) is a bounded domain with sufficiently smooth
boundary, $\nabla$ denotes the gradient, $\diver=\nabla\cdot$ is the divergence and 
$\partial_t=\frac{\partial }{\partial t}$ denotes the partial derivative with respect to time. 
The equations are supplemented by no-flux boundary conditions and suitable initial conditions
\begin{equation}
\label{1.bic}
\begin{array}{rcl}
  (\na u_1 - g(u_1)\na u_2)\cdot\nu &=& 0\\
  (\delta\na u_1 + \kappa \na u_2)\cdot\nu &=& 0\\
\end{array}
  \quad\mbox{on }\partial \Omega,\ t>0,
	\quad u_1(0,x)=u_1^0,\ u_2(0,x)=u_2^0\quad\mbox{in }\Omega,
\end{equation}
where $\nu$ denotes the outer unit normal vector to $\partial \Omega$.
The function $u_1(x,t)\in[0,1]$ represents the density of individuals with
information variable $x\in\Omega$ at time $t\ge 0$, and $u_2(x,t)$ is an influence 
function which modifies the information state of the individuals and possibly may 
lead to a herding (or aggregation) behaviour of individuals. The influence function 
acts through the term $g(u_1)\na u_2$ in \eqref{1.eq1}. The non-negative bounded function 
$g(u_1)$ is assumed to vanish only at $u_1=0$ and $u_1=1$, which provides the bound 
$0\le u_1\le 1$ if $0\leq u_1(0,x)\leq 1$. In particular, we assume that the
influence becomes weak if the number of individuals at fixed $x\in\Omega$
is very low or close to the maximal value $u_1=1$, i.e.\ $g(0)=0$ and $g(1)=0$,
which may enhance herding behaviour. The influence function is assumed to be
modified by diffusive effects also due to the random behaviour of the density
of the individuals with parameter $\delta>0$, by the non-negative source term 
$f(u_1)$, relaxation with time with rate $\alpha>0$, and diffusion with coefficient 
$\kappa>0$.

If $\delta=0$, equations \eqref{1.eq1}-\eqref{1.eq2} can be interpreted
as a nonlinear variant of the chemotaxis Patlak-Keller-Segel model 
\cite{KellerSegel}, where the function $u_2$ corresponds to the 
concentration of the chemoattractant. The model with nonlinear mobility
$g(u_1)$ was first analyzed by Hillen and Painter \cite{HillenPainter},
even for more general mobilities of the type $u_1\beta(u_1)\chi(u_2)$.
Generally, the mobility $g(u_1)=u_1(1-u_1)$ models finite-size exclusion and 
prevents blow-up phenomena \cite{Wrzosek}, which are known in the original 
Keller-Segel model. The convergence to equilibrium was shown in \cite{JiangZhang}.
Such models were also employed to describe evolution of large human crowds driven 
by the dynamic field $u_2$ \cite{BurgerMarkowichPietschmann}.

System \eqref{1.eq1}-\eqref{1.eq2} is one possible model 
to describe the dynamics of information herding in a macroscopic setting. 
There exist other approaches to model herding behaviour, for instance
using kinetic equations \cite{DelitalaLorenzo} or agent-based models \cite{LambdaSeaman}, 
but the focus in this paper is to understand the influence of the parameters $\delta$ 
and $\alpha$ on the solution from a mathematical viewpoint, i.e., to 
investigate the interplay between cross-diffusion and damping.

Equations \eqref{1.eq1}-\eqref{1.eq2} with $\delta>0$ can be derived from an
interacting ``particle'' system modeled by stochastic differential equations, at
least in the case $g(u_1)=\mbox{const.}$ (see \cite{GS14}). One expects that this
derivation can be extended to the case of non-constant $g(u_1)$ but we do not discuss
this derivation here. The above system with $g(u_1)=u_1$ was analyzed in \cite{HittmeirJuengel}
in the Keller-Segel context. The additional cross diffusion with $\delta>0$ in \eqref{1.eq2} 
was motivated by the fact that it prevents the blow up of the solutions in two space dimensions,
even for large initial densities and for arbitrarily small values of $\delta>0$.
The motivation to introduce this term in our model is different since the nonlinear 
mobility $g(u_1)$ allows us to conclude that $u_1\in[0,1]$, thus preventing blow up without 
taking into account the cross-diffusion term $\delta\Delta u_1$. Our aim is to investigate
the solutions to \eqref{1.eq1}-\eqref{1.eq2} for {\em all}
values for $\delta$, thus allowing for {\em destabilizing} cross-diffusion parameters
$\delta<0$. 

One starting point to investigate the dynamics is to consider the functional structure 
of the equation. In this context entropy methods are a possible tool \cite{Juengel1}. 
The entropy structure can frequently be used to establish the existence of (weak) solutions. 
Furthermore, it is helpful for a quantitative analysis of the large-time dynamics of solutions 
for certain reaction-diffusion  systems; see, e.g., \cite{DesvillettesFellner}. The method 
quantifies the decay of a certain functional with respect to a steady state. An advantage is 
that the entropy approach can work globally, even for initial conditions far away from steady states. 
Moreover, the entropy structure may be formulated in the variational framework of gradient flows 
which allows one to analyze the geodesic convexity of their solutions \cite{LiMi13,ZiMa15}. However, 
this global view indicates already that we may not expect that the approach is valid for all 
parameters in general nonlinear systems. Indeed, in many situations, global methods only work 
for a certain range of parameters occurring in the system. The question is what happens for 
parameter values outside the admissible parameter range and near the validity boundary.

One natural conjecture is that upon variation of a single parameter, there exists a \emph{single} 
critical parameter value associated to a first local bifurcation point $\delta_{\textnormal{b}}$ 
beyond which a global functional approach does not extend. In particular, the homogeneous 
steady state upon which the entropy is built, could lose stability and new solution branches may 
appear in parameter space. Another possibility is that global bifurcation branches in parameter space 
are an obstruction. In our context, the generic situation is different from the two natural 
conjectures.

In the context of \eqref{1.eq1}-\eqref{1.eq2}, the \emph{main distinguished parameter} we are 
interested in is $\delta$. Here we shall state our results on an informal level. 
Carrying out the existence of weak solutions and the global decay
to homogeneous steady states 
\benn
u^*=(u_1^*,u_2^*)
\eenn
via an entropy approach, we find the following results:

\begin{enumerate}
 \item[(M1)] Using the entropy approach, one may prove the existence of weak solutions 
 to~\eqref{1.eq1}-\eqref{1.eq2} in certain parameter regimes.
 \item[(M2)] The global entropy decay to equilibrium does not extend to arbitrary 
 negative $\delta$. Suppose we fix all other parameters, then there exists a critical 
 $\delta_{\textnormal{e}}$ (to be defined below) such that global decay occurs only for 
 $\delta>\delta_{\textnormal{e}}$ ($\delta\neq0$).
 \item[(M3)] If we consider the limit $\alpha\ra +\I$ then we can extend the global 
 decay up to
 \benn
 \delta^*:=-\kappa/\gamma<0, \qquad \text{where $\gamma:=\max_{v\in[0,1]} g(v)$,} 
 \eenn
 i.e., global exponential decay to a steady state occurs for all $\delta>\delta^* (\delta\neq 0)$ if 
 $\alpha$ is large enough. 
 \item[(M4)] In the limit $\alpha\ra 0$, we find that 
 $\delta_{\textnormal{e}}\ra +\I$. In particular, 
 the entropy method breaks down in this limiting regime in the formulation presented here.
\end{enumerate}

We stress that the results for the global decay (M2)-(M4) may not be sharp, 
in the sense that one could potentially improve the validity boundary $\delta_{\textnormal{e}}$. 
Interestingly, we shall prove below that (M3) is indeed sharp for certain steady states, 
i.e., no improvement is possible in this limit. The proofs of (M1)-(M4) 
provide a number of technical challenges, which are discussed in more detail in 
Section \ref{sec.results1} and Section \ref{sec.ent.proofs}. We also note that the entropy method 
definitely does not extend to any negative $\delta$. It is clear that a 
global decay to a homogeneous steady state for all initial conditions is impossible if bifurcating 
non-homogeneous steady state solutions exist as well. We use analytical local bifurcation theory 
for the stationary problem, based upon a modification of Crandall-Rabinowitz theory \cite{Kielhoefer}, 
to prove the following:

\begin{itemize}
 \item[(M5)] The bifurcation approach for homogeneous steady states can be carried out as long as 
 \benn
 \delta\neq \delta_{\textnormal{d}}:=-\kappa/g(u_1^*). 
 \eenn
 On a generic open and connected domain, local bifurcations of simple eigenvalues occur for
 \benn
 \delta_{\textnormal{b}}^n= \delta_{\textnormal{d}}+\frac{1}{\mu_n}\Bigl[f'(u_1^*)-
 \frac{\alpha}{g(u_1^*)}\Bigr],
 \eenn
 where $\mu_n$ are the eigenvalues of the negative Neumann Laplacian.
 \item[(M6)] If $\alpha>0$ is sufficiently \emph{large} and fixed, 
 $\delta_\textnormal{b}^n<\delta_\textnormal{d}<\delta^*$ and the bifurcation points accumulate
 at $\delta_\textnormal{d}$.
 \item[(M7)] If $\alpha>0$ is sufficiently \emph{small} and fixed, 
 $\delta_\textnormal{d}<\delta_\textnormal{b}^n$ and the bifurcation points again 
 accumulate at $\delta_\textnormal{d}$.
\end{itemize}

Although these results are completely consistent with the global decay of the entropy functional, 
they do not yield global information about the bifurcation curves. In general, it is 
not possible to analytically characterize all global bifurcation for arbitrary nonlinear systems. 
Therefore, we consider numerical 
continuation of the non-homogeneous steady-state solution branches (for spatial dimension $d=1$). 
The continuation is carried out using \texttt{AUTO} \cite{Doedel_AUTO2007}. Our numerical results 
show the following:

\begin{itemize}
 \item[(M8)] We regularize the numerical problem using a small parameter $\rho$ to avoid 
 higher-dimensional bifurcation surfaces due to mass conservation.
 \item[(M9)] The non-homogeneous steady-state bifurcation branches starting at the local bifurcation points 
 extend in parameter space and contain multi-bump solutions, which deform into more localized (herding) states
 upon changing parameters.
 \item[(M10)] A second continuation run considering $\rho\ra 0$ yields non-trivial solutions for the 
 original problem. In particular, solutions may have multiple transition layers (respectively concentration 
 regions) and the ones with very few layers occupy the largest ranges in $\delta$-parameter space.
\end{itemize}

Combining all the results we conclude that we have the situations in Figure~\ref{fig:00}(a)-(b) for
generic fixed parameter values and a generic fixed domain. These two main cases of interest are:

\begin{itemize}
 \item[(C1)] $\alpha>0$ \emph{sufficiently large}: In this limit, the entropy validity boundary, the 
 analytical bifurcation approach, and the numerical methods are organized around the singular
 limit at $\delta=\delta^*$. Indeed, note that 
 \benn
 \delta^*=\delta_\textnormal{d},\quad \text{if $u=u_1^*$ maximizes $g(u)$ on $[0,1]$},
 \eenn
 and we show below that $\delta_\textnormal{e}\ra \delta^*$ as $\alpha\ra +\I$. The generic 
 picture for a homogeneous steady state so that $u^*_1$ does not maximize $g$ and $\alpha$ is 
 moderate and fixed is given in Figure~\ref{fig:00}(a).
 \item[(C2)] $\alpha>0$ \emph{sufficiently small}: In this case, the generic picture is shown in 
 Figure~\ref{fig:00}(b). The entropy decay only occurs for very large values 
 $\delta>\delta_\textnormal{e}$. Interestingly, the approaches do not seem to collapse onto one 
 singular limit in this case. 
\end{itemize}

We remark that the condition $\kappa\neq-\delta g(u_1)$ does not only occur in the numerical 
continuation analysis. It occurs in the context of the entropy method as well as the analytical 
bifurcation calculation. It is precisely the condition for the vanishing of the determinant of 
the diffusion matrix that prevents pushing existence and decay techniques based upon global 
functionals further. The  condition also prevents analytical bifurcation theory to work as 
the linearized problem does not yield a Fredholm operator. In some sense,
this explains the singular limit as $\alpha\ra +\I$ from (C1). Although (C1) is quite satisfactory 
from a mathematical perspective, one drawback is that the forward problem may not be well-posed 
in a classical sense if $\delta<\delta_\textnormal{d}$; of course, the stationary problem is 
still well-defined.

\begin{figure}[htbp]
\psfrag{u}{\small{$\|u\|$}}
\psfrag{1f}{\tiny{1-layer}}
\psfrag{2f}{\tiny{2-layer}}
\psfrag{d}{\small{$\delta$}}
\psfrag{a}{\small{(a)}}
\psfrag{b}{\small{(b)}}
\psfrag{ds}{\small{$\delta^*$}}
\psfrag{dd}{\small{$\delta_{\textnormal{d}}$}}
\psfrag{de}{\small{$\delta_{\textnormal{e}}$}}
	\centering
		\includegraphics[width=1\textwidth]{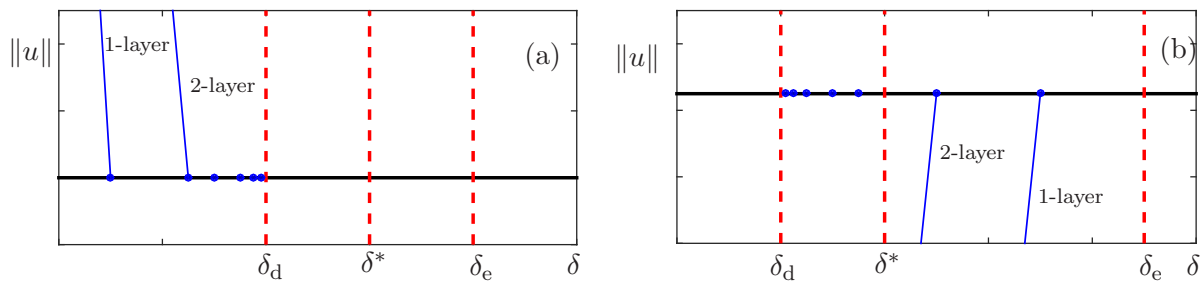}	
		\caption{\label{fig:00}Sketch of the different bifurcation scenarios; for more
		detailed numerical calculations see Section~\ref{sec.numerics.res}. Only 
		the main parameter $\delta$ is varied, a homogeneous branch is shown in black and 
		bifurcation points and branches in blue (dots and curves). Only the first two nontrivial 
		branches are sketched which contain solutions with one transition layer. (a)~Case (C1) 
		with $\alpha>0$ sufficiently large; for a suitable choice of $u^*$ and $\alpha\ra +\I$ 
		all three vertical dashed red lines collapse onto one line. (b)~Case (C2) with $\alpha>0$ 
		sufficiently small.}
\end{figure} 

For (C2), we cannot prove sharp global decay via an entropy functional. However, the first nontrivial 
branch of locally stable stationary herding solutions can be reached in forward time via a classical
well-posed problem, and (C1)-(C2) always make sense for adiabatic parameter variation. 
Although we postpone the detailed mathematical study of the the limit $\alpha\ra0$ to future work, the 
observations raise several interesting problems, which we discuss in the outlook at the end of this paper.

In summary, the main contribution of this work is to study the interplay between three different 
techniques available for reaction-diffusion systems with cross-diffusion: entropy methods, analytical 
local bifurcation and numerical global bifurcation theory. Furthermore, for each technique, we have to
use, improve, and apply the previously available methods to the herding model problem
\eqref{1.eq1}-\eqref{1.bic}. Our results lead to clear insight on the subdivision of parameter
space into regimes, where each method is particularly well-suited to describe the system dynamics.
We identify two interesting singular limits and provide a detailed analysis for the limit of large 
damping. Furthermore, we compute via numerical continuation several solutions that are of interest for 
applications to herding behaviour using a two-parameter homotopy approach to desingularize the mass
conservation. From an application perspective, we identify herding states with clustering of 
individuals in one, or just a few, distinct regions, as the ones occupying the largest parameter
ranges. Hence, we expect applications to be governed by homogeneous stationary and relatively simple 
heterogeneous herding states.

There seem to be very few works \cite{Gabriel,ArnoldBonillaMarkowich} studying the parameter space 
interplay between global entropy-structure methods in comparison to local analytical and global numerical 
bifurcation approaches. Our work seems to be, to the best of our knowledge, the first analysis 
combining and comparing all three methods, and also the first to consider the global-functional and 
bifurcations interaction problem for cross-diffusion systems. In fact, our analysis suggests a general 
paradigm to improve our understanding of global methods for nonlinear spatio-temporal systems, i.e., 
one major goal is to determine the parameter space \emph{validity boundaries} between 
\emph{different methods}.\medskip

The paper is organized as follows. In Section \ref{sec.results}, we state our main results and 
provide an overview of the strategy for the proofs respectively the numerical methods employed. 
In particular, the entropy method results are considered in Section \ref{sec.results1}, the 
analytical local bifurcation in Section \ref{sec.results2}, and the numerical global bifurcation
results in Section \ref{sec.results3}. The following sections contain the full details for 
the main results. The proofs using the entropy method are contained in Section \ref{sec.ent.proofs},
where the weak solution construction is carried out in Section \ref{sec.ex} and the global
decay is proved in Section \ref{sec.decay}. Section \ref{sec.bifurcation.res} proves the 
existence of local bifurcation points to non-trivial solutions upon decreasing $\delta$. The details for 
the global numerical continuation results are reported in Section \ref{sec.numerics.res}. We 
conclude in Section \ref{sec.outlook} with an outlook, where we discuss several open questions.\medskip

{\small \textbf{Notation:} When operating with vectors we view them as column vectors and use 
$(\cdot)^\top$ to denote the transpose. We use the standard notation for $L^p$-spaces, $W^{k,p}$
for the Sobolev space with (weak) derivatives up to and including order $k$ in $L^p$ as well as
the shorthand notation $W^{k,2}=H^k$; see \cite{Evans} for details. Furthermore, $'$ denotes
the associated dual space, when applied to a function space.}

\section{Main Results}
\label{sec.results}

We describe the main results of this paper, obtained by either the entropy method
or local analytical and global numerical bifurcation analysis.

\subsection{Entropy Method}
\label{sec.results1}

First, we show the global existence
of weak solutions and their large-time decay to equilibrium. 
We observe that the diffusion matrix of system  \eqref{1.eq1}-\eqref{1.eq2}
is neither symmetric nor positive definite which complicates the analysis.
Local existence of (smooth) solutions follows from Amann's results
\cite{Ama89} if the system is parabolic in the sense of Petrovskii, i.e., if
the real parts of the eigenvalues of the diffusion matrix are positive. 
A sufficient condition for this
statement is $\delta\ge \delta_{\rm d} = -\kappa/\gamma$. The challenge here
is to prove the existence of {\em global} (weak) solutions.

The main challenge
of \eqref{1.eq1}-\eqref{1.eq2} is that the diffusion matrix of the system is neither
symmetric nor positive definite.
The key idea of our analysis, similar as in \cite{HittmeirJuengel}, 
is to define a suitable entropy functional. 
The entropy is a special Lyapunov functional 
which provides suitable gradient estimates.
Compared to Lyapunov functional techniques like in \cite{Horstmann,Wolansky}
(used for the case $\delta=0$),
the entropy method provides explicit decay rates and, in our case, 
$L^\infty$ bounds without the use of a maximum principle. (Note that in the
system at hand, the $L^\infty$ bounds can be obtained by the standard maximum
principle but there are systems where this can be achieved by using the entropy
method only; see \cite{Juengel1}.)
For this, we introduce the entropy density
\begin{equation*}
  h(u) = h_0(u_1) + \frac{u_2^2}{2\delta_0}, \quad u=(u_1,u_2)^\top\in [0,1]\times\R,
\end{equation*}
where $h_0$ is defined as the second anti-derivative of $1/g$,
\begin{equation}\label{1.h0}
  h_0(s) := \int_m^s\int_m^\sigma\frac{1}{g(t)}~\txtd t~\txtd\sigma, \quad s\in(0,1),
\end{equation}
where $0<m<1$ is a fixed number, and 
$$
  \delta_0:=\delta \quad\mbox{if }\delta>0, \quad
	\delta_0:=\kappa/\gamma \quad\mbox{if }-\kappa/\gamma<\delta<0.
$$
It turns out that the so-called entropy variables
$w=(w_1,w_2)^\top$ with $w_1=h_0'(u_1)$ and $w_2=u_2/\delta_0$ make the diffusion matrix
positive semi-definite for all $\delta>\delta^*:=-\kappa/\gamma$, $\delta\neq 0$. We 
remark that for $\delta=0$ the method does not work and we do not cover this case.
In the $w$-variables, we can formulate \eqref{1.eq1}-\eqref{1.eq2} equivalently as
$$
  \pa_t u = \diver(B(w)\na w) + F(u),
$$
where $u=u(w)$, $F(u)=(0,f(u_1)-\alpha u_2)^\top$ and
\begin{equation}\label{1.B}
  B(w) = \begin{pmatrix}
	g(u_1) & -\delta_0 g(u_1) \\
	\delta g(u_1) & \delta_0\kappa
	\end{pmatrix}.
\end{equation}
The invertibility of the mapping $w\mapsto u(w)$ is guaranteed by Hypothesis (H3)
below. We show in Lemma \ref{lem.B} below that 
$B(w)$ is positive semi-definite if $\delta>\delta^*$, $\delta\neq 0$. The global existence
is based on the fact that the entropy
\begin{equation}\label{1.H}
  H(u(t)) = \int_\Omega \left(h_0(u_1(t)) + \frac{u_2(t)^2}{2\delta_0}\right) \txtd x
\end{equation}
is bounded on $[0,T]$ for any $T>0$; note that we write $u=u(t)$ here to emphasize the time 
dependence of $H$. A formal computation, which is made rigorous in Section 
\ref{sec.ex}, shows that 
\begin{align}
  \frac{\txtd H}{\txtd t}
	&= -\int_\Omega\left(\frac{|\na u_1|^2}{g(u_1)} 
	+ \left(\frac{\delta}{\delta_0}-1\right)\na u_1\cdot\na u_2
	+ \frac{\kappa}{\delta_0}|\na u_2|^2\right)\txtd x \label{1.dHdt} \\
	&\phantom{xx}{}+ \frac{1}{\delta_0}\int_\Omega(f(u_1)-\alpha u_2)u_2~ \txtd x. \nonumber
\end{align}
The terms in the first bracket define a positive definite quadratic form 
if and only if $\delta>\delta^*$. The second integral is bounded since $f(u_1)$ is bounded.
This shows that for some $\eps_1(\delta)>0$,
\begin{equation}\label{1.eps1}
  \frac{\txtd H}{\txtd t} \le -\eps_1(\delta)\int_\Omega\left(
	\frac{|\na u_1|^2}{g(u_1)} 
	+ \frac{|\na u_2|^2}{\delta_0^2}\right)\txtd x + c,
\end{equation}
where the constant $c>0$ depends on $\Omega$, $f$, and $\alpha$. These gradient bounds are 
essential for the existence analysis.

Before we state the existence theorem, we make our assumptions precise:
\begin{enumerate}
\item[(H1)] $\Omega\subset\R^d$ with $\pa\Omega\in C^2$, $\alpha>0$, $\kappa>0$,
$h(u^0)\in L^1(\Omega)$ with $u_1^0\in(0,1)$ a.e. 
\item[(H2)] $f\in C^0([0,1])$ is nonnegative. 
\item[(H3)] $g\in C^2([0,1])$ is positive on $(0,1)$, $g(0)=g(1)=0$, 
$g(u)\le\gamma$ for $u\in[0,1]$, where $\gamma>0$,
and $\int_0^m \txtd s/g(s)=\int_m^1 \txtd s/g(s)=+\infty$ for some $0<m<1$.
\end{enumerate}

The condition $g(u)\le\gamma$ in $[0,1]$ in (H3) implies that 
$(u_1^0-m)^2/(2\gamma)\le h_0(u_1^0)$ and hence, $h(u^0)\in L^1(\Omega)$ in (H1)
yields $u_1^0\in L^2(\Omega)$ and $u_2^0\in L^2(\Omega)$. 
Hypothesis (H3) ensures that the function $h_0$ defined in \eqref{1.h0} 
is well defined and of class $C^4$ (needed in Lemma \ref{lem.csi}).
Its derivative $h_0'$ is strictly increasing on $(0,1)$ with range $\R$, 
thus being invertible with inverse $(h_0')^{-1}:\R\to(0,1)$.
For instance, the function $g(s)=s(1-s)$, $s\in[0,1]$, satisfies (H3) and
$h_0(s) = s\log s + (1-s)\log(1-s)$, where $\log$ denotes the natural logarithm. 
A more 
general class of functions fulfilling (H3) is $g(s)=s^a(1-s)^b$ with $a$, $b\ge 1$.

\begin{theorem}[Global existence]\label{thm.ex}
Let assumptions (H1)-(H3) hold and let $\delta>-\kappa/\gamma$. 
Then there exists a weak 
solution to \eqref{1.eq1}-\eqref{1.bic} satisfying $0\le u_1\le 1$ in 
$\Omega$, $t>0$ and
$$
  u_1,\,u_2\in L^2_{\rm loc}(0,\infty;H^1(\Omega)), \quad
	\pa_t u_1,\,\pa_t u_2\in L^2_{\rm loc}(0,\infty;H^1(\Omega)').
$$
The initial datum is satisfied in the sense of $H^1(\Omega;\R^2)'$. 
\end{theorem}

We provide a brief overview of the proof. First, we discretize the equations in time using 
the implicit Euler scheme, which keeps the entropy structure. Since we are working in the 
entropy-variable formulation, we need to regularize the equations in order to be able to 
apply the Lax-Milgram lemma for the linearized problem. The existence of solutions to the 
nonlinear problem then follows from the Leray-Schauder theorem, where the uniform
estimate is a consequence of the entropy inequality \eqref{1.eps1}. This estimate
also provides bounds uniform in the approximation parameters. A discrete
Aubin lemma in the version of \cite{DreherJuengel} provides compactness, which allows
us to perform the limit of vanishing approximation parameters.

Although the proof is similar to the existence proofs in \cite{HittmeirJuengel,Juengel1},
the results of these papers are not directly applicable since our
situation is more general than in \cite{HittmeirJuengel,Juengel1}. The main novelties
of our existence analysis are the new entropy \eqref{1.H} and the treatment of
destabilizing cross diffusion ($\delta<0$).

For the analysis of the large-time asymptotics, we introduce the constant steady state
$u^*=(u_1^*,u_2^*)$, where 
$$
u^*_1=\overline{u}_1^0,\quad 
u_2^*=\frac{f(u_1^*)}{\alpha},\qquad \text{with }\overline{u}_j^0:=\frac{1}{\m(\Omega)}
\int_\Omega u_j^0(x)~\txtd x,~ j\in\{1,2\},
$$
and $\m(\Omega)$ denotes the Lebesgue measure of $\Omega$. Furthermore, we define 
the relative entropy 
$$
  H(u|u^*) = \int_\Omega h(u|u^*)~\txtd x
$$ 
with the entropy density
\begin{align}\label{1.rel.ent}
  h(u|u^*) = h_0(u_1|u_1^*) + \frac{1}{2\delta_0}(u_2-u_2^*)^2, \quad\mbox{where } 
	h_0(u_1|u^*_1) = h_0(u_1) - h_0(u^*_1).
\end{align}
Note that $u_1$ conserves mass, i.e.\ $\overline{u}_1(t):=\m(\Omega)^{-1}$
$\int_\Omega u_1(t)~\txtd x$ is constant in time and $\overline{u}_1(t)=u_1^*$ for
all $t>0$. Thus, by Jensen's inequality, $h_0(u_1|u_1^*)\ge 0$.

\begin{theorem}[Exponential decay]\label{thm.decay}
Let assumptions (H1)-(H3) hold, let $\Omega$ be convex, let
$f$ be Lipschitz continuous with constant $c_L>0$, and let
\begin{equation}\label{1.eps2}
  \delta_0\eps_1(\delta) > \frac{\gamma}{\alpha}c_L^2 c_S,
\end{equation}
where $\eps_1(\delta)>0$ and $c_S>0$ are defined in Lemmas \ref{lem.B} and 
\ref{lem.csi}, respectively. Then, for $t>0$,
\begin{equation}\label{1.mu}
  H(u(t)|u^*) \le \txte^{-{\chi}(\delta) t}H(u^0|u^*), \quad\mbox{where}\quad
	{\chi}(\delta) := \min\left\{\frac{\eps_1(\delta)}{c_S}
	-\frac{\gamma c_L^2}{\alpha\delta_0},\alpha\right\} > 0.
\end{equation}
Moreover, it holds for $t>0$,
\begin{equation}\label{1.L2decay}
  \|u_1(t)-u^*_1\|_{L^2(\Omega)} + \|u_2(t)-u^*_2\|_{L^2(\Omega)}
	\le 2\sqrt{\max\{\gamma,\delta\}H(u^0|u^*)}\txte^{-{\chi}(\delta) t/2}.
\end{equation}
\end{theorem}

Recall that $\delta_0=\kappa/\gamma$ if $\delta<0$ 
and $\delta_0=\delta$ if $\delta>0$.
The values for $\delta_0\eps_1(\delta)$ are illustrated in Figure~\ref{fig:06}.
It turns out that \eqref{1.eps2} is fulfilled if either the additional
diffusion $\delta>0$ is sufficiently large or if $\gamma/\alpha$ is sufficiently
small. The latter condition means that the influence of the drift term
$g(u_1)\na u_2$ is ``small'' or that the relaxation $-\alpha u_2$ is ``strong''.
The theorem states that in all these cases, the diffusion is sufficiently strong 
to lead to exponential decay to equilibrium. For all parameters fixed, except
$\delta$, we conclude from the condition \eqref{1.eps2} that there exists a 
$\delta_{\textnormal{e}}$ such that exponential decay holds for 
$\delta>\delta_{\textnormal{e}}$ ($\delta\neq 0$) and we see that 
\benn
\lim_{\alpha\ra +\I}\delta_{\textnormal{e}}=\delta^*=-\kappa/\gamma
\eenn
as a singular limit already discussed above. We remark that the exclusion of 
the decay for $\delta=0$ seems to be purely technical and we conjecture that 
exponential decay also holds for $\delta=0$. On the contrary, extensions to
$\alpha\ra0$ are highly nontrivial and we can currently not cover this 
degenerate limiting case using entropy methods.

\begin{figure}[htbp]
\psfrag{d}{\small{$\delta$}}
\psfrag{ds}{\small{$\delta=\delta^*$}}
\psfrag{y}[bl][bl][1][90]{\small{$\delta_0~\epsilon_1(\delta)$}}
	\centering
		\includegraphics[width=0.6\textwidth]{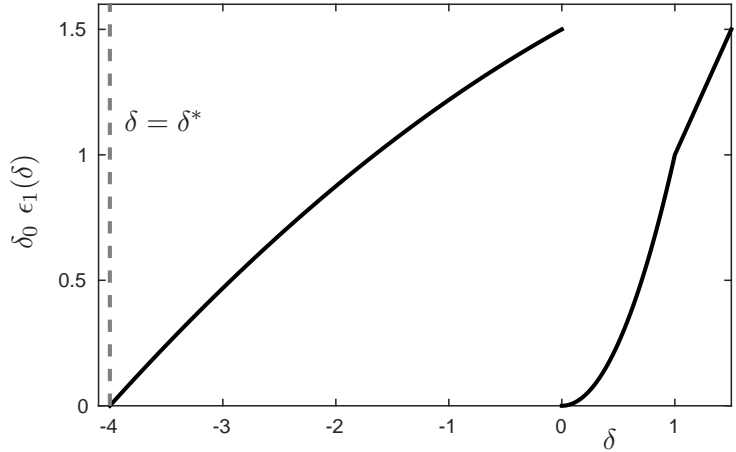}	
		\caption{\label{fig:06}Illustration of $\delta_0\eps_1(\delta)$ for 
		$\kappa=1$ and $\delta=\frac14$ (black curves). The corresponding 
		singular limit $\delta^*=-\kappa/\gamma=-4$ is also marked (grey dashed 
		vertical line).}
\end{figure} 

Theorem \ref{thm.decay} is proved by differentiating the relative entropy $H(u|u^*)$
with respect to time, similar as in \eqref{1.dHdt}. We wish to estimate the
gradient terms from below by a multiple of $H(u|u^*)$. The convex Sobolev
inequality from Lemma \ref{lem.csi} shows that the $L^2$-norm of 
$g(u_1)^{1/2}\na u_1$ is estimated from below by $\int_\Omega h_0(u_1|u_1^*)~\txtd x$, 
up to a factor. The $L^2$-norm of $\na u_2$ is estimated from below by a multiple of
$\int_\Omega(u_2-\overline{u}_2)^2~\txtd x$, using the Poincar\'e inequality. 
However, the variable $u_2$ generally does not conserve mass and in particular,
$\overline{u}_2\neq u_2^*$. We exploit instead the relaxation term in \eqref{1.eq2} 
to achieve the estimate
$$
  H(u(t)|u^*) + {\chi}(\delta) \int_0^t H(u(s)|u^*)~\txtd s \le 0.
$$
Then Gronwall's lemma gives the result. The difficulty is the estimate of
the source term $f(u_1)$. This problem is overcome by controlling the expression
involving $f(u_1)$ by taking into account the contribution coming from
the convex Sobolev inequality. However, we need that $\delta$ is sufficiently large,
i.e., cross diffusion has to dominate reaction.

The above arguments hold on a formal level only. A second difficulty is to make
these arguments rigorous since we need the test function $h_0'(u_1)-h_0'(u_1^*)$,
which is undefined if $u_1=0$ or $u_1=1$ (since $h_0'(0)=-\infty$ and
$h_0'(1)=+\infty$ by Hypothesis (H3)). 
The idea is to perform a transformation of variables in terms of
so-called entropy variables which ensure that $0<u_1<1$ in a time-discrete
setting. Passing from the semi-discrete to the continuous case, the variable
$u_1$ may satisfy $0\le u_1\le 1$ in the limit. 

\subsection{Analytical Bifurcation Analysis}
\label{sec.results2}

As outlined in the introduction, the first natural conjecture for the failure of the entropy 
method is to study bifurcations of the homogeneous 
steady states $u^*=(u_1^*,u_2^*)$, which solve
\begin{equation}
\label{eq:steady}
\begin{array}{l}
  0 = \diver(\na u_1 - g(u_1)\na u_2), \\
  0 = \diver(\delta\na u_1 + \kappa \na u_2) + f(u_1)-\alpha u_2,
\end{array}
\end{equation}
with the no-flux boundary conditions \eqref{1.bic}. To 
study the bifurcations of $u^*$ under variation of $\delta$ we use the right-hand side
of \eqref{eq:steady} to define a bifurcation function and apply the theory of Crandall-Rabinowitz
\cite{CrandallRabinowitz,Kielhoefer}. The problem is that $u^*$ is not an \emph{isolated} bifurcation
branch as a function of $\delta$ since fixing any initial mass yields a different one-dimensional
family of homogeneous steady states with
\begin{equation}
\label{eq:mass}
u_1^* = \frac{1}{\m(\Omega)}\int_\Omega u_1(x)~\txtd x\geq 0.
\end{equation}
Hence, the standard approach has to be modified and we follow arguments that can be found in
\cite{ChertockKurganovWangWu,ShiWang,WangXu}. It is helpful to introduce some notations first.
For $p>d$, let 
\be
\label{eq:space_ba}
\begin{array}{lcl}
\mathcal{X}&:=&\{ u\in W^{2,p}(\Omega): \nabla u \cdot \nu=0\textnormal{ on }\partial\Omega\},\\
\mathcal{Y}&:=&L^p(\Omega),\\
\mathcal{Y}_0&:=&\left\{u_1\in L^p(\Omega):\int_\Omega u_1(x)~\txtd x=0\right\}\\
\end{array}
\ee
where the space $\mathcal{X}$ includes standard Neumann boundary conditions. 
Due to the Sobolev embedding theorem we know that $W^{2,p}(\Omega)$ is continuously embedded 
in $C^{(1+\theta)}(\bar{\Omega})$ for some $\theta\in(0,1)$. 
If Neumann boundary conditions hold, then our original boundary conditions~\eqref{1.bic} hold as 
well. However, the converse is only true if we can invert the diffusion matrix, i.e., as long 
as $\delta\neq \delta_d=-\frac{\kappa}{g(u_1)}$. In particular, we shall always assume for the 
local bifurcation analysis of homogeneous steady states that 
\be
\label{eq:snondegen}
\delta\neq \delta_\textnormal{d}=-\frac{\kappa}{g(u_1^*)}.
\ee
This implies that me may not find all possible bifurcations and the single point when
the diffusion matrix vanishes has to be treated separately; we leave this as a goal for
future work.

Next, we define the mapping $\mathcal{F}:\mathcal{X}\times\mathcal{X}\times\R\longrightarrow\mathcal{Y}_0\times\mathcal{Y}\times 
\R$ by
\begin{equation}\label{F}
	\mathcal{F}(u_1,u_2,\delta):=
	\begin{pmatrix}
	\diver(\nabla u_1- g(u_1)\nabla u_2) \\
	\delta\Delta u_1+\kappa\Delta u_2-\alpha u_2+f(u_1) \\
	\int_\Omega u_1(x)~\txtd x - \m(\Omega) u_1^* \\
	\end{pmatrix}.
\end{equation}
The first two terms are the usual bifurcation functions one would naturally define, the third 
term is used to isolate the bifurcation branch for the mapping $\cF$, i.e., to avoid the 
problem with mass conservation, while the last two terms take into account the boundary 
conditions. We know that there exists a family of homogeneous steady state solutions 
$$
\cF(u_1^*,u_2^*,\delta)=0
$$
for each $\delta\in\R$. The goal is to find the parameter values $\delta_{\textnormal{b}}$ such 
that at $\delta=\delta_{\textnormal{b}}$ a non-trivial (or non-homogeneous) branch of steady 
states is generated at the bifurcation point; see also Figure~\ref{fig:00}. We are going to 
check that $\cF$ is $C^1$-smooth and the Fr\'echet derivative $\txtD_u\mathcal{F}$ with respect 
to $u$ at a point $\tilde{u}=(\tilde{u}_1,\tilde{u}_2)$ is given by
\begin{equation} \label{eq:DF}
  \cA_\delta(\tilde{u})\begin{pmatrix} U_1\\ U_2\end{pmatrix}:=\txtD_u\mathcal{F}(\tilde{u},\delta)
  \begin{pmatrix} U_1\\ U_2\end{pmatrix}=
  \begin{pmatrix}
  \Delta U_1-\diver[ g'(\tilde{u}_1)(\nabla \tilde{u}_2)U_1+g(\tilde{u}_1)\nabla U_2] \\
  \delta\Delta U_1+\kappa\Delta U_2-\alpha U_2+f'(\tilde{u}_1)U_1 \\
	\int_\Omega U_1(x)~\txtd x\\
  \end{pmatrix}
\end{equation}
where $(U_1,U_2)^\top\in\cX\times \cX$ and $\cA_\delta:\cX\times \cX\ra \cY_0\times \cY\times \R$. 
We already know from Theorem \ref{thm.decay} that for all $\delta>\delta_{\txte}$ ($\delta\neq 0$), 
the homogeneous steady state $u^*$ is globally stable. Clearly this implies local stability as well 
and this fact can also be checked by studying the spectrum of $\cA_\delta(u^*)$. From the structure 
of the cross-diffusion equations \eqref{1.eq1}-\eqref{1.eq2} one does expect destabilization of the 
homogeneous state upon decreasing $\delta$.\medskip

\begin{theorem}
 \label{thm:main_ana_bif}
Let $u^*=(u_1^*,u_2^*)$ be a homogeneous steady state, consider the generic parameter case with
$-\kappa\neq \delta g(u_{1}^*)$ and suppose all eigenvalues $\mu_n$ of the negative Neumann 
Laplacian on $\Omega$ are simple. Then the following hold:
 \begin{itemize}
 \item[(R1)] $\txtD_u\cF(\tilde{u},\delta):\mathcal{X}
 \times\mathcal{X}\ra\mathcal{Y}_0 \times\mathcal{Y}\times\R$ is a Fredholm 
 operator with index zero;
 \item[(R2)] there exists a sequence of bifurcation points $\delta=\delta^n_{\textnormal{b}}$ 
 such that $\dim\left(\cN[\txtD_u\mathcal{F}(u^*,\delta^n_{\textnormal{b}})]\right)=1$, 
 where $\cN[\cdot]$ denotes the nullspace;
 \item[(R3)] there exist simple real eigenvalues $\lambda_n(\delta)$ of $\cA_\delta(u^*)$, 
 which satisfy $\lambda_n(\delta^n_\textnormal{b})=0$. Furthermore, $\lambda_n(\delta)$ crosses 
 the imaginary axis at $\delta^n_{\textnormal{b}}$ with non-zero speed, i.e.,
 $\txtD_{\delta u}F(u^*,\delta^n_{\textnormal{b}})e^n_{\textnormal{b}}\notin 
 \cR[\cA_{\delta^n_{\textnormal{b}}}]$, where $\cR[\cdot]$ denotes the range and 
 $\textnormal{span}[e^n_{\textnormal{b}}]=\cN[\cA_{\delta^n_{\textnormal{b}}}]$.
\end{itemize}
\end{theorem}

The results from (R1)-(R3) hold quite generically (i.e., for $\delta\neq \delta_{\txtd}$ and 
for generic domains~\cite{Uhlenbeck}) and yield, upon applying a standard result by 
Crandall-Rabinowitz \cite{CrandallRabinowitz,CrandallRabinowitz1,Kielhoefer}, the existence 
of branches of non-trivial solutions 
$$
(u_1[s],u_2[s],\delta[s])\in\cX\times \cX\times \R,\qquad 
(u_1[0],u_2[0],\delta[0])=(u_1^*,u_2^*,\delta^n_ {\textnormal{b}}),
$$
where $s\in[-s_0,s_0]$ parametrizes the steady-state branch locally for some small $s_0>0$, 
and $(u_1[s],u_2[s],\delta[s])\neq (u_1^*,u_2^*,\delta^n_{\textnormal{b}})$ for 
$s\in[-s_0,0)\cup(0,s_0]$. Slightly more 
precise information about the branch can be obtained using the eigenfunction $e_{\textnormal{b}}$ and 
we refer to Section \ref{sec.bifurcation.res} for the details. The main conclusion of the bifurcation 
theorem is that we know that the entropy method cannot show the decay to steady state for all 
parameter regions. However, to track the non-trivial solution branches in parameter space, it is 
usually not possible to compute the global shape of all bifurcation branches analytically.
In this case, numerical bifurcation analysis is extremely helpful.

\subsection{Numerical Bifurcation Analysis}
\label{sec.results3}

The results from Section \ref{sec.results1}-\ref{sec.results2} do not provide a full exploration 
of the dynamical structure of the solutions for the parameter regime $\delta<\delta^*$. To understand 
this regime better we study the bifurcations of \eqref{eq:steady} numerically for
\be
\label{eq:fixed_numerics}
f(s)=s(1-s),\qquad g(s)=s(1-s),\qquad s\in\Omega=[0,l]\subset \R.
\ee
for some interval length $l>0$. Note that this yields a boundary-value problem (BVP) involving 
two second-order ordinary differential equations (ODEs)
\bea
0&=&\frac{\txtd}{\txtd x}\left(\frac{\txtd u_1}{\txtd x}-g(u_1)\frac{\txtd u_2}{\txtd x}\right),
\label{eq:ODE1_bad}\\
\label{eq:ODE2_bad} 
0&=&\delta\frac{\txtd^2 u_1}{\txtd x^2}+\kappa \frac{\txtd^2 u_2}{\txtd x^2}-\alpha u_2+f(u_1).
\eea
with boundary conditions 
\bea
0&=&\frac{\txtd u_1}{\txtd x}(0)-g(u_1(0))\frac{\txtd u_2}{\txtd x}(0),\qquad 
0=\delta \frac{\txtd u_1}{\txtd x}(0)+\kappa\frac{\txtd u_2}{\txtd x}(0), \label{eq:ODE3_bad}\\
0&=&\frac{\txtd u_1}{\txtd x}(1)-g(u_1(1))\frac{\txtd u_2}{\txtd x}(1),\qquad
0=\delta \frac{\txtd u_1}{\txtd x}(1)+\kappa\frac{\txtd u_2}{\txtd x}(1). \label{eq:ODE4_bad}
\eea
An excellent available tool to study the problem \eqref{eq:ODE1_bad}-\eqref{eq:ODE4_bad} is 
the software \texttt{AUTO} \cite{Doedel_AUTO2007} for numerical continuation of BVPs; for other 
possible options and extensions we refer to the discussion in Section \ref{sec.outlook}. 
\texttt{AUTO} is precisely designed to deal with BVPs for ODEs of the form
\be
\label{eq:BVP_AUTO}
\frac{\txtd z}{\txtd x}=F(z;p),\qquad x\in[0,1],\quad G(w(0),w(1))=0
\ee
where $F:\R^N\times \R^P\ra \R^N$, $G:\R^N\times \R^N\ra \R^N$ and $p\in\R^P$ are 
parameters and $z=z(x)\in\R^N$ is the unknown vector. It is easy to re-write 
\eqref{eq:ODE1_bad}-\eqref{eq:ODE4_bad} as a system in the form \eqref{eq:BVP_AUTO} of 
four first-order ODEs, i.e., we get $N=4$, consider the scaling $\tilde{x}=x/l$ to normalize 
the interval length to one, then drop the tilde for $x$ again, and let 
$$
p_1:=\delta,\quad p_2:=\kappa,\quad p_3:=\alpha,\quad p_4:=l, 
$$
so $P=4$ with primary bifurcation parameter $\delta$. For more background on \texttt{AUTO} and 
on numerical continuation we refer to \cite{KrauskopfOsingaGalan-Vioque,Keller,Govaerts}. In the 
setup \eqref{eq:BVP_AUTO} one can numerically continue the family of homogeneous solutions 
$$
(u^*,\delta)=(u_1^*,u_2^*,\delta)
$$
as a function of $\delta$, i.e., to compute $u^*=u^*(\cdot;\delta)$ for $\delta$ in some 
specified parameter interval. Although this 
calculation yields bifurcation points for some $\delta$ values, it is not straightforward
to use the formulation \eqref{eq:ODE1_bad}-\eqref{eq:ODE2_bad} to switch onto the 
non-homogeneous solution branches generated at the bifurcation point. The problem is
due to the mass conservation since 
$$
\overline{u}_1=\m(\Omega)^{-1}\int_\Omega u_1~\txtd x=u_1^*, \qquad 
u_2^*=\frac{f(u_1^*)}{\alpha}
$$
is a solution for every positive initial mass $\overline{u}_1^0$. In particular, the branch
of solutions is not isolated and there exist parametric two-dimensional families of solutions. There 
are multiple ways to deal with this problem; see also Section \ref{sec.outlook}. One possibility is 
to resolve the degeneracy of the problem via a small parameter $0<\rho\ll1$ and consider  
\bea
0&=&\frac{\txtd}{\txtd x}\left(\frac{\txtd u_1}{\txtd x}-g(u_1)\frac{\txtd u_2}{\txtd x}\right)-
\rho(u_1-\overline{u}_1),
\label{eq:ODE1}\\
0&=&\delta\frac{\txtd^2 u_1}{\txtd x^2}+\kappa \frac{\txtd^2 u_2}{\txtd x^2}-\alpha u_2+f(u_1).
\label{eq:ODE2}
\eea
for a fixed positive parameter $\overline{u}_1>0$. In particular, upon setting 
$$
z_1:=u_1,\qquad z_2:=u_2, \qquad z_3:=\frac{\txtd u_1}{\txtd x},\qquad z_4:=\frac{\txtd u_2}{\txtd x}, 
$$
as well as 
$$
p_5:=\overline{u}_1,\qquad p_6:=\rho,\qquad P=6,
$$
we end up with a problem of the form \eqref{eq:BVP_AUTO} by transforming the two second-order ODEs 
to four first-order ODEs and re-labelling parameters. The vector field for the ODE-BVP 
we study numerically is then given by 
\be
\label{eq:vf}
F(z;p)=
\begin{pmatrix}
p_4 z_3 \\
p_4 z_4 \\
p_4[-g(z_1)f(z_1)+p_3g(z_1)z_2 + p_2 g'(z_1)z_3z_4+p_2p_6(z_1-p_5)]/\cD_g\\
p_4[-f(z_1) + p_3z_2 - p_1g'(z_1)z_3z_4 - p_1(z_1-p_5)p_6]/\cD_g\\
\end{pmatrix}
\ee
where $\cD_g:=p_2+p_1g(z_1)$ and the detailed choices for the free parameters are discussed in 
Section \ref{sec.numerics.res}. Observe that the system \eqref{eq:vf} becomes singular if 
$\cD_g=0$, which is precisely the condition $\delta\neq -\kappa/g(u_1)$ already 
discovered above. Therefore, we would need also for the numerical analysis a re-formulation
(or de-singularization) of the problem to deal with this singularity and we postpone this 
problem to future work. As mentioned above, the primary bifurcation parameter we are 
going to be interested in is $\delta=p_1$. The main results of the numerical bifurcation analysis, 
which are presented in full detail in Section \ref{sec.numerics.res}, are the following:

\begin{itemize}
 \item[(B1)] As predicted by the analytical results, we find the existence of local bifurcation 
points on the branch of homogeneous steady states in the parameter region with 
$\delta<\delta_{\textnormal{d}}$ for the case of sufficiently large $\alpha$ and 
for $\delta>\delta_{\textnormal{d}}$ for the case of sufficiently small $\alpha$.
At each bifurcation point on the homogeneous branch, a simple eigenvalue crosses the imaginary axis.
 \item[(B2)] The non-trivial (i.e.~non-homogeneous) solution branches consist of solutions of 
multiple 'interfaces' or 'layers'; branches originating further away from $\delta_{\textnormal{d}}$ 
contain less layers. The branches can acquire sharper layers upon variation of further parameters
which is important for information herding.
 \item[(B3)] At the local bifurcation points, we observe the emergence of two symmetric branches 
 of solutions for the case when the nonlinearities are identical quadratic nonlinearities
 of the form $s\mapsto s(1-s)$.
 \item[(B4)] We also construct non-homogeneous solutions for $\rho=0$ by a homotopy continuation 
 step first continuing onto the non-trivial branches in $\delta$ and then decreasing $\rho$ to zero
 in a second continuation step.
 \item[(B5)] Furthermore, we also study the shape deformation of non-trivial solutions upon
variation of $\kappa$ and the domain length $l$. The numerical results show that the main interesting
structures of the problem have already been obtained by just varying $\delta$ and $\alpha$.
\end{itemize}

\section{Entropy Method -- Proofs}
\label{sec.ent.proofs}

\subsection{Proof of Theorem \ref{thm.ex}}\label{sec.ex}

First, we prove that the new diffusion matrix $B(w)$, defined in \eqref{1.B},
is positive semi-definite if $\delta$ is not too negative.

\begin{lemma}\label{lem.B}
Assume (H3) and $\delta>-\kappa/\gamma$, $\delta\neq 0$. 
Then the matrix $B(w)$ is positive semi-definite, and there exists 
$\eps_1(\delta)>0$ such that for all $z=(z_1,z_2)^\top\in\R^2$, $w\in\R^2$:
$$
  z^\top B(w)z \ge \eps_1(\delta)(g(u_1)z_1^2 + z_2^2).
$$
It holds $\eps_1(\delta)\to 0$ as $\delta\searrow 0$ and 
$\delta\searrow -\kappa/\gamma$ (see Figure~\ref{fig:05a}).
\end{lemma}

\begin{figure}[htbp]
\psfrag{d}{\small{$\delta$}}
\psfrag{ds}{\small{$\delta=\delta^*$}}
\psfrag{y}[bl][bl][1][90]{\small{$\epsilon_1(\delta)$}}
	\centering
		\includegraphics[width=0.6\textwidth]{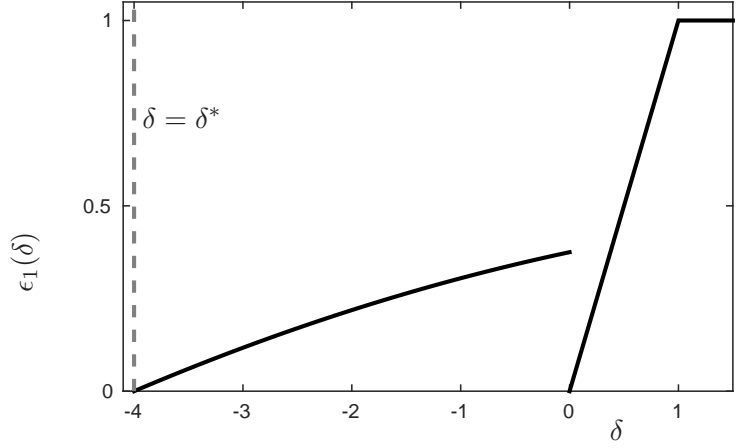}	
		\caption{\label{fig:05a}Illustration of $\eps_1(\delta)$ for 
		$\kappa=1$ and $\delta=\frac14$ (black curves). The corresponding 
		singular limit $\delta^*=-\kappa/\gamma=-4$ is also marked (grey dashed 
		vertical line).}
\end{figure} 

For later use, we note that the lemma implies that
\begin{equation}\label{ex.wBw}
  \na w:B(w)\na w \ge \eps_1(\delta)\left(\frac{|\na u_1|^2}{g(u_1)}
	+ \frac{|\na u_2|^2}{\delta_0^2}\right),
\end{equation}
where $w=(w_1,w_2)=(h_0'(u_1),u_2/\delta_0)$ are the entropy variables
introduced in the introduction and $A:B=\sum_{i,j}A_{ij}B_{ij}$ for
two matrices $A=(A_{ij})$, $B=(B_{ij})$.

\begin{proof}
Let $z=(z_1,z_2)^\top\in\R^2$. Then
$$
  z^\top B(w)z = g(u_1)z_1^2 - (\delta_0-\delta)g(u_1) z_1z_2 
	+ \delta_0\kappa z_2^2.
$$
If $\delta>0$, then $\delta_0=\delta$ and the mixed term vanishes, showing the claim
for $\eps_1(\delta)=\min\{1,\delta\kappa\}$.
If $-\kappa/\gamma<\delta<0$, we have $\delta_0=\kappa/\gamma$. 
We make the (non-optimal) choice
$$
  \eps_0=\eps_0(\delta) = \frac12\left(1 - \frac14\left(1-\frac{\gamma\delta}{\kappa}
	\right)^2\right) > 0.
$$
Then $\eps_0<1-(1-\gamma\delta/\kappa)^2/4$, which is equivalent to
$(\kappa-\gamma\delta)^2<4(1-\eps_0)\kappa^2$. Thus, using $g(u_1)\le\gamma$
(see assumption (H3)),
\begin{align*}
  z^\top B(w)z &= g(u_1)z_1^2 - \left(\frac{\kappa}{\gamma}-\delta\right)g(u_1) z_1z_2 
	+ \frac{\kappa^2}{\gamma} z_2^2 \\
	&= \eps_0 g(u_1)z_1^2	
	+ (1-\eps_0)g(u_1)\left(z_1 - \frac{(\kappa-\gamma\delta)z_2}{2\gamma(1-\eps_0)}
	\right)^2 \\
	&\phantom{xx}{}
	+ \frac{1}{\gamma}\left(\kappa^2 
	- \frac{(\kappa-\gamma\delta)^2}{4\gamma(1-\eps_0)}g(u_1)
	\right)z_2^2 \\
	&\ge \eps_0 g(u_1)z_1^2	+  \frac{1}{\gamma}\left(\kappa^2 
	- \frac{(\kappa-\gamma\delta)^2}{4(1-\eps_0)}\right)z_2^2.
\end{align*}
In view of the choice of $\eps_0$, the bracket on the right-hand side is positive,
and the claim follows after choosing 
$\eps_1(\delta)=\min\{\eps_0(\delta),
[\kappa^2-(\kappa-\gamma\delta)^2/(4(1-\eps_0(\delta)))]/\gamma\}>0$ for
$-\kappa/\gamma<\delta<0$.
\end{proof}

The proof of Theorem \ref{thm.ex} is based on the solution of a time-discrete 
and regularized problem.

{\em Step 1: Solution of an approximate problem.}
Let $T>0$, $N\in\N$, $\tau=T/N$, $\eps>0$, and $n\in\N$ such that $n>d/2$.
Then $H^n(\Omega;\R^2)\hookrightarrow L^\infty(\Omega;\R^2)$.
Let $w^{k-1}\in L^\infty(\Omega;\R^2)$ be given. If $k=1$, we define
$w^0=h'(u^0)$.  
We wish to find $w^k\in H^n(\Omega;\R^2)$ such that
\begin{align}\label{ex.disc}
  \frac{1}{\tau}\int_\Omega & (u(w^k)-u(w^{k-1}))\cdot\phi ~\txtd x
	+ \int_\Omega\na\phi:B(w^k)\na w^k ~\txtd x \\
	&{}+ \eps\int_\Omega\bigg(\sum_{|\beta|=n}D^\beta w^k\cdot D^\beta\phi 
	+ w^k\cdot\phi\bigg)~\txtd x = \int_\Omega F(u(w^k))\cdot\phi ~\txtd x \nonumber
\end{align}
for all $\phi\in H^n(\Omega;\R^2)$, where $\beta\in\N_0^n$ is a multi-index,
$D^\beta$ is the corresponding partial derivative, $u(w)=(h')^{-1}(w)$ for $w\in\R$,
and we recall that $F(u)=(0,f(u_1)-\alpha u_2)^\top$.
By definition of $h_0$, we find that $u_1(w)\in(0,1)$, thus avoiding any
degeneracy at $u_1=0$ or $u_1=1$.

The existence of a solution to
\eqref{ex.disc} will be shown by a fixed-point argument. In order to
define the fixed-point operator, let $y\in L^\infty(\Omega;\R^2)$ and $\eta\in[0,1]$
be given. We solve the linear problem
\begin{equation}\label{ex.LM}
  a(w,\phi) = G(\phi)\quad\mbox{for all }\phi\in H^n(\Omega;\R^2),
\end{equation}
where 
\begin{align*}
  a(w,\phi) &= \int_\Omega\na\phi:B(y)\na w ~\txtd x + \eps\int_\Omega
	\left(\sum_{|\beta|=n}D^\beta w\cdot D^\beta\phi + w\cdot\phi\right)~\txtd x, \\
	G(\phi) &= -\frac{\eta}{\tau}\int_\Omega\big(u(y)-u(w^{k-1})\big)~\txtd x
	+ \eta\int_\Omega F(u(y))\cdot\phi ~\txtd x.
\end{align*}
The forms $a$ and $G$ are bounded on $H^n(\Omega;\R^2)$. Moreover, in view of the
positive semi-definiteness of $B(y)$ and the generalized Poincar\'e inequality
(see Chap.~II.1.4 in \cite{Temam}), the bilinear form $a$ is coercive:
$$
  a(w,w) \ge \eps\int_\Omega\bigg(\sum_{|\beta|=n}|D^\beta w|^2 + |w|^2\bigg)~\txtd x
	\ge \eps c\|w\|_{H^n(\Omega)} \quad\mbox{for }w\in H^n(\Omega;\R^2).
$$
By the Lax-Milgram lemma, there exists a unique solution $w\in H^n(\Omega;\R^2)
\hookrightarrow L^\infty(\Omega;\R^2)$ to \eqref{ex.LM}. This defines the fixed-point
operator $S:L^\infty(\Omega;\R^2)\times[0,1]\to L^\infty(\Omega;\R^2)$, 
$S(y,\eta)=w$. 

By construction, $S(y,0)=0$ for all $y\in L^\infty(\Omega;\R^2)$, and standard
arguments show that $S$ is continuous and compact, observing that the embedding
$H^n(\Omega;\R^2)\hookrightarrow L^\infty(\Omega;\R^2)$ is compact. 
It remains to prove a uniform
bound for all fixed points of $S(\cdot,\eta)$. Let $w\in L^\infty(\Omega;\R^2)$ be such
a fixed point. Then $w$ solves \eqref{ex.LM} with $y$ replaced by $w$. With the test function
$\phi=w$, we find that
\begin{align}
  \frac{\eta}{\tau}\int_\Omega & (u(w)-u(w^{k-1}))\cdot w ~\txtd x 
	+ \int_\Omega \na w:B(w)\na w ~\txtd x \label{ex.aux1} \\
	&{}+ \eps\int_\Omega\bigg(\sum_{|\beta|=n}|D^\beta w|^2 
	+ |w|^2\bigg)~\txtd x= \eta\int_\Omega F(u(w))\cdot w ~\txtd x. \nonumber
\end{align}
Since $h_0''=1/g>0$ on $(0,1)$, $h_0$ is convex. Consequently,
$h_0(x)-h_0(y)\le h_0'(x)(x-y)$ for all $x$, $y\in[0,1]$.
Choosing $x=u(w)$ and $y=u(w^{k-1})$ and using $h_0'(u(w))=w$, this gives
$$
  \frac{\eta}{\tau}\int_\Omega(u(w)-u(w^{k-1}))\cdot w ~\txtd x
	\ge \frac{\eta}{\tau}\int_\Omega\big(h(u(w))-h(u(w^{k-1}))\big)~\txtd x.
$$
Since $u_1=u_1(w)\in(0,1)$ and $f$ is continuous, 
there exists $f_M=\max_{s\in[0,1]}f(s)$ and thus,
$$
  \int_\Omega F(u(w))\cdot w ~\txtd x \le \int_\Omega(f_M-\alpha u_2)u_2 ~\txtd x 
	\le c_f,
$$
where $c_f>0$ only depends on $f_M$ and $\alpha$.
Hence, \eqref{ex.aux1} can be estimated as follows:
\begin{align}\label{ex.aux2}
  \eta\int_\Omega h(u(w))~\txtd x 
	&{}+ \tau\int_\Omega \na w:B(w)\na w ~\txtd x
	+ \eps\tau\int_\Omega\bigg(\sum_{|\beta|=n}|D^\beta w|^2 
	+ |w|^2\bigg)~\txtd x \\
	&\le \eta\tau c_f + \eta\int_\Omega h(u(w^{k-1}))~\txtd x. \nonumber
\end{align}
This yields an $H^n$ bound for $w$ uniform in $\eta$ 
(but not uniform in $\tau$ or $\eps$).
The Leray-Schauder fixed-point theorem shows the existence of a solution 
$w\in H^n(\Omega;\R^2)$
to \eqref{ex.LM} with $y$ replaced by $w$ and with $\eta=1$, 
which is a solution to \eqref{ex.disc}.

{\em Step 2: Uniform bounds.} Let $w^k$ be a solution to \eqref{ex.disc}.
Set $w^{(\tau)}(x,t)=w^k(x)$ and $u^{(\tau)}(x,t)=u(w^k(x))$ for $x\in\Omega$ and
$t\in((k-1)\tau,k\tau]$, $k=1,\ldots,N$. At time $t=0$, we set 
$w^{(\tau)}(\cdot,0)=h_0'(u^0)$
and $u^{(\tau)}(0)=u^0$. We introduce the shift operator 
$(\sigma_\tau u^{(\tau)})(t)
= u(w^{k-1})$ for $t\in((k-1)\tau,k\tau]$, $k=1,\ldots,N$. Then $u^{(\tau)}$ solves
\begin{align}
  \frac{1}{\tau}\int_0^T\int_\Omega & (u^{(\tau)} - \sigma_\tau u^{(\tau)})\cdot\phi ~\txtd x~\txtd t
	+ \int_0^T\int_\Omega\na\phi:B(w^{(\tau)})\na w^{(\tau)} ~\txtd x~\txtd t \label{ex.aux3} \\
	&{}+ \eps\int_0^T\int_\Omega\bigg(\sum_{|\beta|=n}D^\beta w^{(\tau)}\cdot D^\beta\phi
	+ w^{(\tau)}\cdot\phi\bigg)~\txtd x~\txtd t = \int_0^T\int_\Omega F(u^{(\tau)})\cdot\phi ~\txtd x~\txtd t 
	\nonumber
\end{align}
for piecewise constant functions $\phi:(0,T)\to H^n(\Omega;\R^2)$. 
By density, the weak formulation
also holds for all $L^2(0,T;H^n(\Omega;\R^2))$. 

We have shown in Step 1 that the solution $w=w^k$ satisfies the entropy 
estimate \eqref{ex.aux2}. By \eqref{ex.wBw}, we obtain the gradient estimate
$$
  \int_\Omega\na w^k:B(w^k)\na w^k ~\txtd x
	\ge \eps_1(\delta)\min\{\gamma^{-1},\delta_0^{-2}\}\int_\Omega
	(|\na u_1^k|^2 + |\na u_2^k|^2)~\txtd x,
$$
since $g(u_1^k)\le\gamma$. Thus, we obtain from \eqref{ex.aux2} the following 
entropy inequality:
\begin{align}\label{ex.ent}
  \int_\Omega h(u^k)~\txtd x &+ c_0\tau\int_\Omega(|\na u_1^k|^2 + |\na u_2^k|^2)~\txtd x \\
	&+ \eps\tau\int_\Omega\bigg(\sum_{|\beta|=n}|D^\beta w^k|^2 + |w^k|^2\bigg)~\txtd x 
	\le c_f\tau + \int_\Omega h(u^{k-1})~\txtd x, \nonumber
\end{align}
where $c_0=\eps_1(\delta)\min\{\gamma^{-1},\delta_0^{-2}\}$. 
Adding these inequalities leads to
\begin{align*}
  \int_\Omega h(u^k)~\txtd x &+ c_0\tau\sum_{j=1}^k \int_\Omega(|\na u_1^j|^2 
	+ |\na u_2^j|^2)~\txtd x \\
	&+ \eps\tau\sum_{j=1}^k\int_\Omega\bigg(\sum_{|\beta|=n}|D^\beta w^k|^2 
	+ |w^k|^2\bigg)~\txtd x 
	\le c_f k\tau + \int_\Omega h(u^{0})~\txtd x.
\end{align*}
Since
$$
  \int_\Omega h(u^k)~\txtd x = \int_\Omega\left(h_0(u_1^k) 
	+ \frac{(u_2^k)^2}{2\delta_0}\right)~\txtd x
	\ge \frac{1}{2\delta_0}\int_\Omega(u_2^k)^2 ~\txtd x,
$$
the above estimate shows the following uniform bounds:
\begin{align}
  \|u_1^{(\tau)}\|_{L^\infty(0,T;L^\infty(\Omega))}
	+ \|u_2^{(\tau)}\|_{L^\infty(0,T;L^2(\Omega))} &\le c, \label{ex.Linfty} \\
	\|u_1^{(\tau)}\|_{L^2(0,T;H^1(\Omega))} + \|u_2^{(\tau)}\|_{L^2(0,T;H^1(\Omega))} 
	&\le c, 
	\label{ex.H1} \\
	\sqrt{\eps}\|w^{(\tau)}\|_{L^2(0,T;H^n(\Omega))} &\le c, \label{ex.eps}
\end{align}
where $c>0$ denotes here and in the following a constant which is
independent of $\eps$ or $\tau$ (but possibly depending on $T$).

In order to derive a uniform estimate for the discrete time derivative, let
$\phi\in L^2(0,T;$ $H^n(\Omega))$. Then, setting $Q_T=\Omega\times(0,T)$,
\begin{align}
  \frac{1}{\tau}&\left|\int_\tau^T\int_\Omega(u_1^{(\tau)}-\sigma_\tau u_1^{(\tau)})\phi 
	~\txtd x~\txtd t\right| 
	\le \big(\|\na u_1^{(\tau)}\|_{L^2(Q_T)}
	+ \|g(u_1^{(\tau)})\|_{L^\infty(Q_T)}\|\na u_2^{(\tau)}\|_{L^2(Q_T)}\big) \nonumber \\
	&\phantom{xx}{}\times\|\na\phi\|_{L^2(Q_T)} 
	+ \eps\|w_1^{(\tau)}\|_{L^2(0,T;H^n(\Omega))}\|\phi\|_{L^2(0,T;H^n(\Omega))} 
	\nonumber \\
	&\le c\sqrt{\eps}\|\phi\|_{L^2(0,T;H^n(\Omega))} + c\|\phi\|_{L^2(0,T;H^1(\Omega))}, 
	\nonumber \\
	\frac{1}{\tau}&\left|\int_\tau^T\int_\Omega(u_2^{(\tau)}-\sigma_\tau u_2^{(\tau)})\phi 
	~\txtd x~\txtd t\right| 
	\le \big(\delta\|\na u_1^{(\tau)}\|_{L^2(Q_T)}
	+ \kappa\|\na u_2^{(\tau)}\|_{L^2(Q_T)}\big)\|\na\phi\|_{L^2(Q_T)} \label{ex.aux4} \\
	&\phantom{xx}{}+ \eps\|w_1^{(\tau)}\|_{L^2(0,T;H^n(\Omega))}
	\|\phi\|_{L^2(0,T;H^n(\Omega))}
	+ \big(\|f(u_1^{(\tau)})\|_{L^2(Q_T)} + \alpha\|u_2^{(\tau)}\|_{L^2(Q_T)}\big)
	\|\phi\|_{L^2(Q_T)} \nonumber \\
	&\le c\sqrt{\eps}\|\phi\|_{L^2(0,T;H^n(\Omega))} 
	+ c\|\phi\|_{L^2(0,T;H^1(\Omega))}, \nonumber 
\end{align}
which shows that
\begin{equation}\label{ex.tau}
  \tau^{-1}\|u^{(\tau)}-\sigma_\tau u^{(\tau)}\|_{L^2(0,T;(H^n(\Omega))')} \le c.
\end{equation}

{\em Step 3: The limit $(\eps,\tau)\to 0$.} The uniform estimates \eqref{ex.H1} and
\eqref{ex.tau} allow us to apply the discrete Aubin lemma in the version of 
\cite{DreherJuengel},
showing that, up to a subsequence which is not relabelled, as $(\eps,\tau)\to 0$,
\begin{align}
  u^{(\tau)} \to u &\quad\mbox{strongly in }L^2(0,T;L^2(\Omega))
	\mbox{ and a.e. in }Q_T, \label{conv.u} \\
	u^{(\tau)} \rightharpoonup u &\quad\mbox{weakly in }L^2(0,T;H^1(\Omega)), 
	\nonumber \\
	\tau^{-1}(u^{(\tau)}-\sigma_\tau u^{(\tau)}) \rightharpoonup \pa_t u
	&\quad\mbox{weakly in }L^2(0,T;(H^n(\Omega))'), \nonumber \\
	\eps w^{(\tau)} \to 0 &\quad\mbox{strongly in }L^2(0,T;H^n(\Omega)). \nonumber
\end{align}
Because of the $L^\infty$ bound \eqref{ex.Linfty} for $(u_1^{(\tau)})$, we have
$$
  g(u_1^{(\tau)}) \rightharpoonup^* g(u_1), \quad 
	f(u_1^{(\tau)}) \rightharpoonup^* f(u_1) 
	\quad\mbox{weakly* in }L^\infty(0,T;L^\infty(\Omega))
$$
(and even strongly in $L^p(Q_T)$ for any $p<\infty$).
Thus, we can pass to the limit $(\eps,\tau)\to 0$ in \eqref{ex.aux3} to obtain
a solution to
\begin{align*}
  \int_0^T\langle\pa_t u_1,\phi\rangle ~\txtd t 
	+ \int_0^T\int_\Omega(\na u_1-g(u_1)\na u_2)\phi ~\txtd x~\txtd t
	&= 0, \\
	\int_0^T\langle\pa_t u_2,\phi\rangle ~\txtd t 
	+ \int_0^T\int_\Omega(\delta\na u_1+\kappa\na u_2)\phi ~\txtd x~\txtd t
	&= \int_0^T\int_\Omega(f(u_1)-\alpha u_2)\phi ~\txtd x~\txtd t
\end{align*}
for all $\phi\in L^2(0,T;H^n(\Omega))$. In fact, performing the limit $\eps\to 0$ 
and then
$\tau\to 0$, we see from \eqref{ex.aux4} that $\pa_t u\in L^2(0,T;(H^1(\Omega))')$ and
hence, the weak formulation also holds for all $\phi\in L^2(0,T;H^1(\Omega))$.
It contains the no-flux boundary conditions \eqref{1.bic}. 
Moreover, the initial conditions are satisfied
in the sense of $(H^1(\Omega;\R^2))'$; see Step 3 of the proof of Theorem 2 in
\cite{Juengel1}. 
This finishes the proof.

\subsection{Proof of Theorem \ref{thm.decay}}\label{sec.decay}

We recall first the following convex Sobolev inequality which is used to estimate
the gradient terms in the entropy inequality.

\begin{lemma}\label{lem.csi}
Let $\Omega\subset\R^d$ ($d\ge 1$) be a convex domain and let $\phi\in C^4$ be a
convex function such that $1/\phi''$ is concave. Then there exists $c_S>0$
such that for all integrable
functions $u$ with integrable $\phi(u)$ and $\phi''(u)|\na u|^2$,
$$
  \frac{1}{\m(\Omega)}\int_\Omega\phi(u)~\txtd x
	- \phi\bigg(\frac{1}{\m(\Omega)}\int_\Omega u~\txtd x\bigg)
	\le \frac{c_S}{\m(\Omega)}\int_\Omega\phi''(u)|\na u|^2 ~\txtd x,
$$
where $\m(\Omega)$ denotes the measure of $\Omega$.
\end{lemma}

\begin{proof}
The lemma is a consequence of Prop.~7.6.1 in \cite{BakryGentilLedoux} after choosing
the probability measure $d\mu=\txtd x/\m(\Omega)$ on $\Omega$ and the differential operator
$L=\Delta - x\cdot\na$, which satisfies the curvature condition $C\!D(1,\infty)$
since $\Gamma_2(u)=\frac12(|\na^2 u|^2+|\na u|^2)\ge \frac12|\na u|^2 = \Gamma(u)$.
Another proof can be found in \cite[Section 3.4]{AMTU01}.
\end{proof}

{\em Step 1: Uniform bound for the $L^1$ norm of $u_1^k$.}
The $L^1$ norm of $u_1^k$ is not conserved but we are able to control its $L^1$ norm.
For this, let $w^k\in H^n(\Omega;\R^2)$ be a solution to \eqref{ex.disc} 
and set $u_1^k=u_1(w^k)$.
We introduce the notation $\overline{v}=\m(\Omega)^{-1}\int_\Omega v(x)~\txtd x$ 
for any integrable function $v$. This implies that $u^*_1=\overline{u}_1^0$.
Employing the test function $\phi=(1,0)$ in \eqref{ex.disc}, we find that
$\overline{u}_1^k = \overline{u}_1^{k-1} - \eps\tau \overline{w}_1^k$. 
Solving the recursion gives
$$
  \overline{u}_1^k = \overline{u}^0_1 - \eps\tau\sum_{j=1}^k\overline{w}_1^j
	= u^*_1 - \eps\tau\sum_{j=1}^k\overline{w}_1^j,
$$
and by \eqref{ex.eps}, we conclude that
$$
  |\overline{u}_1^{(\tau)}(t)-u^*_1| \le \eps\|w_1^{(\tau)}\|_{L^1(0,t;L^1(\Omega))}
	\le \sqrt{\eps}c,
$$
where $\overline{u}^{(\tau)}_1(t)=\overline{u}_1^k$ for $t\in((k-1)\tau,k\tau]$.
Consequently, as $(\eps,\tau)\to 0$, the convergence \eqref{conv.u}
shows that $\overline{u}_1(t)=u^*_1$ for $t>0$. 

{\em Step 2: Estimate of the relative entropy.} We employ the test function
$$
  \phi=(h_0'(u_1^k)-h_0'(u^*_1),(u_2^k-u^*_2)/\delta_0)
  =(w_1^k-h_0'(u^*_1),w_2^k-u^*_2/\delta_0)
$$ 
in \eqref{ex.disc} to obtain
\begin{align}
  0 &= \frac{1}{\tau}\int_\Omega \left((u_1^k-u_1^{k-1})
	(h_0'(u_1^k)-h_0'(u^*_1))
	+ \frac{1}{\delta_0}(u_2^k-u_2^{k-1})(u_2^k-u^*_2)\right)~\txtd x \nonumber \\
	&\phantom{xx}{}+ \int_\Omega \na w^k:B(w^k)\na w^k ~\txtd x 
	+ \eps\int_\Omega\bigg(\sum_{|\beta|=n}|D^\beta w^k|^2 
	+ w_1^k(w_1^k-h_0'(u^*_1)) \label{time.aux} \\
	&\phantom{xx}{}+ w_2^k(w_2^k-u^*_2)/\delta_0)\bigg)~\txtd x
	- \frac{1}{\delta_0}\int_\Omega (f(u_1^k)-\alpha u_2^k)(u_2^k-u^*_2)~\txtd x 
	\nonumber \\
	&=: I_1 + \cdots + I_4. \nonumber
\end{align}
For the first integral, we employ the convexity of $h_0$:
\begin{align*}
  (u_1^k-u_1^{k-1})(h_0'(u_1^k)-h_0'(u^*_1))
	&\ge (h_0(u_1^k) - h_0(u_1^{k-1})) - h_0'(u^*_1)(u_1^k-u_1^{k-1}), \\
	(u_2^k-u_2^{k-1})(u_2^k-u^*_2)
	&\ge \frac12\big((u_2^k-u^*_2)^2 - (u_2^{k-1}-u^*_2)^2\big),
\end{align*}
which yields
\begin{align*}
  I_1 &\ge \frac{1}{\tau}\int_\Omega (h_0(u_1^k) - h_0(u_1^{k-1}))~\txtd x
	- \frac{h_0'(u^*_1)}{\tau}\int_\Omega (u_1^k-u_1^{k-1})~\txtd x \\
	&\phantom{xx}{}
	+ \frac{1}{2\delta_0\tau}\int_\Omega\big((u_2^k-u^*_2)^2-(u_2^{k-1}-u^*_2)^2\big)~\txtd x.
\end{align*}
By \eqref{ex.wBw}, it follows that
$$
  I_2 \ge \eps_1(\delta)\int_\Omega\left(\frac{|\na u_1^k|^2}{g(u_1^k)}
	+ \frac{|\na u_2^k|^2}{\delta_0^2}\right)~\txtd x 
	= \eps_1(\delta)\int_\Omega\left(h_0''(u_1^k)|\na u_1^k|^2
	+ \frac{|\na u_2^k|^2}{\delta_0^2}\right)~\txtd x. 
$$
Lemma \ref{lem.csi} then shows that
$$
  I_2 \ge \frac{\eps_1(\delta)}{c_S}\int_\Omega(h_0(u_1^k) 
	- h_0(\overline{u}_1^k))~\txtd x
  + \frac{\eps_1(\delta)}{\delta_0^2}\int_\Omega|\na u_2^k|^2 ~\txtd x. 
$$
The third integral in \eqref{time.aux} is estimated by using Young's inequality:
$$
  I_3 \ge \frac{\eps}{2}\int_\Omega\big((w_1^k)^2 + (w_2^k)^2 - h_0'(u^*_1)^2 
	- \delta_0^{-2}(u^*_2)^2\big)~\txtd x
	\ge -\frac{\eps}{2}\int_\Omega\big(h_0'(u^*_1)^2 
	+ \delta_0^{-2}(u^*_2)^2\big)~\txtd x.
$$

Summarizing these estimates, we infer from \eqref{time.aux} that
\begin{align*}
  &\int_\Omega (h_0(u_1^k)-h_0(u_1^{k-1}))~\txtd x
	- h_0'(u^*_1)\int_\Omega (u_1^k-u_1^{k-1})~\txtd x \\
	&\phantom{xx}{}
	+ \frac{1}{2\delta_0}\int_\Omega\big((u_2^k-u^*_2)^2-(u_2^{k-1}-u^*_2)^2\big)~\txtd x \\
	&\phantom{xx}{}
	+ \frac{\eps_1(\delta)\tau}{c_S}\int_\Omega (h_0(u_1^k)-h_0(\overline{u}_1^k))~\txtd x 
  + \frac{\eps_1(\delta)\tau}{\delta_0^2}\int_\Omega|\na u_2^k|^2 ~\txtd x \\
	&\le \frac{\eps\tau}{2}\int_\Omega\big(h_0'(\overline{u}^k_1)^2 
	+ \delta_0^{-2}(u^*_2)^2\big)~\txtd x
	+ \frac{\tau}{\delta_0}\int_\Omega(f(u_1^k)-\alpha u_2^k)(u_2^k-u^*_2)~\txtd x.
\end{align*}
Adding these equations over $k$ and using the notation as in the proof
of Theorem \ref{thm.ex} for $u_i^{(\tau)}$, we obtain
\begin{align}
  & \int_\Omega(h_0(u_1^{(\tau)}(t)) - h_0(u_1^0))~\txtd x
	- h_0'(u^*_1)\int_\Omega(u_1^{(\tau)}(t)-u_1^0)~\txtd x \nonumber \\
	&\phantom{xx}{}+ \frac{1}{2\delta_0}\int_\Omega\big((u_2^{(\tau)}(t)-u^*_2)^2
	- (u_2^0-u^*_2)^2\big)~\txtd x \label{time.aux2} \\
	&\phantom{xx}{}+ \frac{\eps_1(\delta)}{c_S}\int_0^t\int_\Omega\big(h_0(u_1^{(\tau)})
	- h_0(\overline{u}_1^{(\tau)})\big)~\txtd x ~\txtd s
	+ \frac{\eps_1(\delta)}{\delta_0^2}
	\int_0^t\int_\Omega|\na u_2^{(\tau)}|^2 ~\txtd x~\txtd s \nonumber \\
	&\le \frac{\eps}{2}\int_0^t\int_\Omega\big(h_0'(\overline{u}_1^{(\tau)})^2
	+ \delta_0^{-2}(u^*_2)^2\big)~\txtd x~\txtd s 
	+ \frac{1}{\delta_0}\int_0^t\int_\Omega(f(u_1^{(\tau)})-\alpha u_2^{(\tau)})
	(u_2^{(\tau)}-u^*_2)~\txtd x~\txtd s. \nonumber
\end{align}

{\em Step 3: The limit $(\eps,\tau)\to 0$.}
Because of the $L^\infty$ bound for $(u_1^{(\tau)})$, it follows that, 
for a subsequence,
$u_1^{(\tau)}\rightharpoonup^* u_1$ weakly* in $L^\infty(0,T;L^1(\Omega))$ 
and thus, as $(\eps,\tau)\to 0$,
$$
  \int_\Omega(u_1^{(\tau)}(t)-u_1^0)~\txtd x = \int_\Omega(u_1^{(\tau)}(t)-u^*_1)~\txtd x
	\to \int_\Omega(u_1(t)-u^*_1)~\txtd x = 0,
$$
since $\overline{u}_1(t)=u^*_1$ for $t>0$, by Step 1.
The weak convergence of $(\na u_2^{(\tau)})$ to $\na u_2$ in 
$L^2(0,T;L^2(\Omega))$ implies that
$$
  \liminf_{\tau\to 0}\int_0^t\int_\Omega|\na u_2^{(\tau)}|^2 ~\txtd x ~\txtd s
	\le \int_0^t\int_\Omega|\na u_2|^2 ~\txtd x ~\txtd s.
$$
Furthermore, by the strong convergence $u_1^{(\tau)}\to u_1$ in $L^2(0,T;L^2(\Omega))$,
up to a subsequence, $u_1^{(\tau)}\to u_1$ a.e.\ in $Q_T=\Omega\times(0,T)$ and
$h_0(u_1^{(\tau)})\to h_0(u_1)$ a.e.\ in $Q_T$. Then the 
$L^\infty$ bound of $(u_1^{(\tau)})$ implies that
$h_0(u_1^{(\tau)})\to h_0(u_1)$ strongly in $L^p(0,T;L^p(\Omega))$ for any $p<\infty$.
Furthermore, we know that $u_2^{(\tau)}\to u_2$ strongly in $L^2(0,T;L^2(\Omega))$,
see \eqref{conv.u}.
Therefore, the limit $(\eps,\tau)\to 0$ in \eqref{time.aux2} leads to
\begin{align*}
  & \int_\Omega(h_0(u_1(t)) - h_0(u_1^0))~\txtd x
	+ \frac{1}{2\delta_0}\int_\Omega\big((u_2(t)-u^*_2)^2
	- (u_2^0-u^*_2)^2\big)~\txtd x \\
	&\phantom{xx}{}+ \frac{\eps_1(\delta)}{c_S}\int_0^t\int_\Omega\big(h_0(u_1)
	- h_0(u^*_1)\big)~\txtd x ~\txtd s
	+ \frac{\eps_1(\delta)}{\delta_0^2}\int_0^t\int_\Omega|\na u_2|^2 ~\txtd x~\txtd s \nonumber \\
	&\le \frac{1}{\delta_0}\int_0^t\int_\Omega(f(u_1)-\alpha u_2)(u_2-u^*_2)~\txtd x~\txtd s.
\end{align*}

Now, we estimate the right-hand side.
Because of $f(u^*_1)=\alpha u^*_2$ and the Lipschitz continuity of $f$ with
Lipschitz constant $c_L>0$, we infer that (recall \eqref{1.rel.ent} for the definition
of $h_0(u_1|u^*_1)$)
\begin{align*}
  &\int_\Omega\big(h_0(u_1(t)|u^*_1)~\txtd x - h_0(u_1(0)|u^*_1)\big)~\txtd x
	+ \frac{1}{2\delta_0}\int_\Omega\big((u_2(t)-u^*_2)^2 - (u_2(0)-u^*_2)^2\big)~\txtd x \\
	&\phantom{xx}{}+ \frac{\eps_1(\delta)}{c_S}\int_0^t\int_\Omega 
	h_0(u_1(s)|u^*_1)~\txtd x~\txtd s \\
	&\le \frac{1}{\delta_0}\int_0^t\int_\Omega(f(u_1)-f(u^*_1))(u_2-u^*_2)~\txtd x~\txtd s
	- \frac{\alpha}{\delta_0}\int_0^t\int_\Omega(u_2-u^*_2)^2 ~\txtd x~\txtd s \\
	&\le \frac{1}{2\delta_0\alpha}\int_0^t\int_\Omega(f(u_1)-f(u^*_1))^2 ~\txtd x~\txtd s 
	- \frac{\alpha}{2\delta_0}\int_0^t\int_\Omega(u_2-u^*_2)^2 ~\txtd x~\txtd s\\
	&\le \frac{c_L^2}{2\alpha\delta_0}\int_0^t\int_\Omega(u_1-u^*_1)^2 ~\txtd x~\txtd s
	- \frac{\alpha}{2\delta_0}\int_0^t\int_\Omega(u_2-u^*_2)^2 ~\txtd x~\txtd s.
\end{align*}
Since $\overline{u}_1=u^*_1$, a Taylor expansion and the assumption 
$1/h_0''(u_1)=g(u_1)\le\gamma$ give
\begin{align}
  \int_0^t\int_\Omega h_0(u_1|u^*_1)~\txtd x~\txtd s
  &= \int_0^t\int_\Omega(h_0(u_1)-h_0(u^*_1)~\txtd x~\txtd s \nonumber \\
	&= \int_0^t\int_\Omega\left(h_0'(u^*_1)(u_1-u^*_1) 
	+ \frac12h_0''(\xi)(u_1-u^*_1)^2\right)~\txtd x~\txtd s  \label{time.h0} \\
	&\ge \frac{1}{2\gamma}\int_0^t\int_\Omega (u_1-u^*_1)^2 ~\txtd x~\txtd s, \nonumber
\end{align}
where $\xi$ is a number between $u_1$ and $u_1^*$. We conclude that
\begin{align*}
  &\int_\Omega h_0(u_1(t)|u^*_1)~\txtd x
	+ \frac{1}{2\delta_0}\int_\Omega(u_2(t)-u^*_2)^2 ~\txtd x 
	+ \left(\frac{\eps_1(\delta)}{c_S}-\frac{\gamma c_L^2}{\alpha\delta_0}\right)
	\int_0^t\int_\Omega h_0(u_1(s)|u^*_1)~\txtd x~\txtd s \\
	&\phantom{xx}{}
	+ \frac{\alpha}{2\delta_0}\int_0^t\int_\Omega(u_2-u^*_2)^2 ~\txtd x~\txtd s 
	\le \int_\Omega h_0(u_1(0)|u^*_1)~\txtd x 
	+ \frac{1}{2\delta_0}\int_\Omega(u_2(0)-u^*_2)^2~\txtd x,
\end{align*}
and recalling the notation $h(u|U)=h_0(u_1|u^*_1)+(u_2-u^*_2)^2/(2\delta_0)$,
$$
  \int_\Omega h(u(t)|U)~\txtd x 
	+ \min\left\{\frac{\eps_1(\delta)}{c_S}-\frac{\gamma c_L^2}{\alpha\delta_0},
	\alpha\right\}\int_0^t h(u|U)~\txtd s
	\le \int_\Omega h(u(0)|U)~\txtd x.
$$
Then Gronwall's lemma implies that
$$
  H(u(t)|U) = \int_\Omega h(u(t)|U)~\txtd x \le e^{-{\chi}(\delta) t}H(u(0)|U), \quad t\ge 0,
$$
where ${\chi}(\delta)$ is defined in \eqref{1.mu}.
Finally, taking into account \eqref{time.h0}, we estimate
$$
  h(u|U) \ge \frac{1}{2\gamma}(u_1-u^*_1)^2 + \frac{1}{2\delta}(u_2-u^*_2)^2,
$$
which shows \eqref{1.L2decay} and finishes the proof.

\section{Analytical Bifurcation Analysis -- Proofs}\label{sec.bifurcation.res}

In this section, we are going to prove Theorem \ref{thm:main_ana_bif}. The proofs follow 
closely ideas presented for similar systems in \cite{ChertockKurganovWangWu,ShiWang,WangXu}, 
which are fundamentally based upon an application of results of Crandall and Rabinowitz 
\cite{CrandallRabinowitz,CrandallRabinowitz1}; see also \cite{Kielhoefer} for a detailed 
exposition of the these results. Recall that we defined the spaces $\cX$, $\cY$, $\cY_0$
in \eqref{eq:space_ba} and the mapping 
$$
\cF:\cX\times \cX\times \R\ra \cY_0\times \cY\times \R
$$ 
in \eqref{F}. A first step is to investigate the Fredholm and differentiability properties 
of $\cF$.

\begin{lemma}
 \label{propF}
 The mapping $\cF$ satisfies the following properties:
 \begin{itemize}
 \item[(L1)] $\cF(u^*,\delta)=0$ for all $\delta\in\R$.
 \item[(L2)] $\cF(u_1,u_2,\delta)=0$ implies that $(u_1,u_2)$ solves \eqref{eq:steady}.
 \item[(L3)] $\cF$ is $C^1$-smooth with Fr\'echet derivative $\txtD_u\cF$ given by \eqref{eq:DF}.
 \item[(L4)] If $\tilde{u}(x)\equiv (\tilde{u}_1,\tilde{u}_2)$ is a homogeneous state and 
 $\delta g(\tilde{u}_{1})\neq -\kappa$ then $\txtD_u\cF(\tilde{u}_1,\tilde{u}_2,\delta)$ 
 is a Fredholm operator with index zero.
 \end{itemize}
\end{lemma}

\begin{proof}
For (L1) recall that $u^*=(u_1^*,u_2^*)$ was the notation for a homogeneous steady state.
Regarding (L2), observe that the first two components of $\cF$ are just the steady state
equations \eqref{eq:steady}. Statement (L3) follows from a direct calculation. The problem is to
show (L4). We follow the argument given in \cite{ChertockKurganovWangWu,WangXu} and consider
\begin{equation}
\label{DF2}
\txtD_{u}\mathcal{F}(\tilde{u}_1,\tilde{u}_2,\delta)(U_1,U_2)^\top
=\cB_1(U_1,U_2)^\top+\cB_2(U_1,U_2)^\top,
\end{equation}
where $\cB_1:\mathcal{X}\times\mathcal{X}\ra\mathcal{Y}_0
\times\mathcal{Y}\times\R$
is defined by
\begin{equation}
\label{F1}
	\cB_1 \begin{pmatrix} U_1\\ U_2 \end{pmatrix}=
	\begin{pmatrix}
	\Delta U_1-\diver[ g'(\tilde{u}_1)(\nabla \tilde{u}_2)U_1+g(\tilde{u}_1)\nabla U_2] \\
	\delta\Delta U_1+\kappa\Delta U_2-\alpha U_2+f'(\tilde{u}_1)U_1 \\
	0\\
	\end{pmatrix},
\end{equation}
and the mapping $\cB_2:\mathcal{X}\times\mathcal{X}\ra\mathcal{Y}_0
\times\mathcal{Y}\times\R$ is given by
\begin{equation}
\label{F2}
	\cB_2 \begin{pmatrix} U_1 \\ U_2 \\ \end{pmatrix}=
	\begin{pmatrix}
	0\\
	0 \\
	\int_\Omega U_1(x)~\txtd x\\
	\end{pmatrix}.
\end{equation}
We observe easily that $\mathcal{B}_2:\mathcal{X}\times\mathcal{X}
\ra\mathcal{Y}_0\times\mathcal{Y}\times\R$ is linear and compact.
We need an ellipticity condition and $\cB_1$ should satisfy Agmon's condition~\cite{ShiWang}. 
We have ellipticity for $\cB_1$ (in the sense of Petrovskii~\cite{Jan98,ShiWang}) if
\be
\label{eq:maindet}
\det\left[
\begin{pmatrix}
       1 & -g(\tilde{u}_1)\\
       \delta & \kappa
       \end{pmatrix}\xi\cdot\xi\right]
\neq 0,
\ee
for all $\xi=(\xi_1,\xi_2,\ldots,\xi_d)\in\R^{d}\backslash\{0\}$.
Computing the determinant this condition just yields
\benn
0\neq (\xi_1^2+\dots+\xi_d^2)(\kappa+\delta g(\tilde{u}_1))\quad 
\text{if and only if}\quad -\kappa\neq\delta g(\tilde{u}_1)
\eenn
and ellipticity in the sense of Petrovskii follows. Moreover we need to verify 
Agmon's condition at a fixed angle $\theta\in[-\pi,\pi)$. 
Using~\cite[Remark~2.5]{ShiWang} with $\theta=\pi/2$, one verifies computing 
a shifted determinant similar to the previously computed one in \eqref{eq:maindet} 
that Agmon's condition holds for all values of $\kappa$. In particular, the 
ellipticity condition gives a restriction on the parameters for the bifurcation 
analysis and not Agmon's condition. By applying~\cite[Thm.~3.3]{ShiWang} we infer that 
\benn
\cB_1:\mathcal{X}\times\mathcal{X}\ra\mathcal{Y}
\times\mathcal{Y}\times\{0\}
\eenn
is a Fredholm operator of index zero. Hence
$\mathcal{Y}_0\times\mathcal{Y}\times\{0\}=\mathcal{R}(\cB_1)\oplus W$,
where $\mathcal{R}(\cB_1)$ is the range of $\cB_1$ and $W$ is a closed subspace of 
$\mathcal{Y}\times\mathcal{Y}\times\R$ with $\dim W=\dim \mathcal{N}(\cB_1)<\infty$.
Consequently, since the first component of $\cB_1$ is in $\mathcal{Y}_0$, we have
\benn
\mathcal{Y}_0\times\mathcal{Y}\times\R = 
\mathcal{R}(\cB_1)\oplus W_0\oplus\text{span}\{(0,0,1)^\top\}
\eenn
where $W_0=\{(H_1,H_2,H_3)\in W|\int_0^LH_1(x)dx=0\}$ and $W=W_0+\text{span}\{(1,0,0)\}$. 
Then $\dim W=\dim W_0+1$.
Thus the codimension of $\mathcal{R}(\cB_1)$ in $\mathcal{Y}_0\times\mathcal{Y}\times\R$
is equal to $\dim W=\dim \mathcal{N}(\cB_1)$. Hence, $\cB_1:\cX\times \cX 
\ra \cY_0\times \cY\times \R$ is a Fredholm operator of index zero 
for $\delta g(\tilde{u}_{1})\neq -\kappa$. Therefore, $\txtD_u\mathcal{F}$ is a Fredholm
operator of index zero as $\cB_2$ is a compact perturbation. Hence, the result (R1) in 
Theorem~\ref{thm:main_ana_bif} follows.
\end{proof}

It seems difficult to improve the result to include the degenerate cases when 
$\kappa=-\delta g(u_1^*)$ as this would require to deal with bifurcation problems with
non-elliptic operators. The next goal is to apply~\cite[Thm.~4.3]{ShiWang}. To do so, we 
need some additional 
properties of $\mathcal{F}$. In particular, in order that bifurcations occur from the 
homogeneous steady state  $u^*=(u_1^*,u_2^*)$  we need that the implicit function theorem 
fails. For the following lemma we need to be in the case where each eigenvalue $\mu_n$ of 
the negative Neumann Laplacian on $\Omega$ eigenvalue is simple. For the one-dimensional case this 
always holds, while for generic $d$-dimensional domains the eigenvalues are also 
simple~\cite{Uhlenbeck}.

\begin{lemma}
\label{cond_RC}
 Suppose the eigenvalues of the negative Neumann Laplacian on $\Omega\subset \R^d$ are simple and 
 $\delta g(u_1^*)\neq -\kappa$. Then there exist bifurcation points at $\delta=
 \delta^n_{\textnormal{b}}$ such that the map $\mathcal{F}$ satisfies the following properties:
 \begin{itemize}
  \item[(L5)] the null space $\mathcal{N}[\txtD_u\mathcal{F}(u^*,\delta^n_{\textnormal{b}})]$ 
  is one-dimensional, i.e., $\textnormal{span}[e^n_{\textnormal{b}}]=
  \mathcal{N}[\txtD_u\mathcal{F}(u^*,\delta^n_{\textnormal{b}})]$;
  \item[(L6)] the non-degenerate crossing condition holds, i.e.,
  \begin{equation}
\label{eq:cross_cond}
 \txtD_{\delta u}\mathcal{F}(u^*,\delta^n_{\textnormal{b}})e^n_{\textnormal{b}}\notin 
 \cR[\txtD_u\mathcal{F}(u^*,\delta^n_{\textnormal{b}})].
\end{equation}
 \end{itemize}
\end{lemma}

\begin{proof}
We start by proving (L5). By~\eqref{F1}, the null space of $\txtD_u\mathcal{F}(u^*,\delta)$ consists 
of solutions for
\begin{equation}
\label{eq:null}
\begin{aligned}
  \Delta U_1-g(u_1^*)\Delta U_2 &=0, \\
	\delta\Delta U_1+\kappa\Delta U_2-\alpha U_2+f'(u_1^*)U_1 &=0, \\
	\int_\Omega U_1(x)~\txtd x &=0,
\end{aligned}
\end{equation}
with no-flux conditions on $\partial\Omega$. For 
any pair $u=(u_1,u_2)\in\mathcal{X}\times\mathcal{X}$, we can expand $u_1$ and $u_2$ as 
a series of mutually orthogonal eigenfunctions of the following system
\begin{equation}
\label{eq:lapl}
	\begin{split}
	\left\{
	\begin{array}{rcll}
	\displaystyle
	-\Delta u&=&\mu u& \qquad \text{in}\quad\Omega, \\
	\frac{\partial u}{\partial\nu}&=&0 &\qquad \text{on}\quad\partial\Omega,
	\end{array}
	\right.
	\end{split}
\end{equation}
multiplied by constants vectors. Let $\mu_n>0$ be a simple eigenvalue of~\eqref{eq:lapl} and 
$e_{\mu_n}$ is the eigenfunction corresponding to $\mu_n$ normalized by 
$\int_\Omega (e_{\mu_n})^2 ~\txtd x=1$. Then we define
\benn
\bar{U}_1:= \int_\Omega u_1(x)e_{\mu_n}(x)~\txtd x,\qquad 
\bar{U}_2:=\int_\Omega u_2(x)e_{\mu_n}(x)~\txtd x.
\eenn
We obtain
\begin{equation}
\label{UV}
\begin{aligned}
\int_\Omega e_{\mu_n}\Delta u_1 ~\txtd x=-\mu_n\int_\Omega u_1 e_{\mu_n} ~\txtd x
=-\mu_n \bar{U}_1,\\
\int_\Omega e_{\mu_n}\Delta u_2~ \txtd x
=-\mu_n\int_\Omega u_2 e_{\mu_n}~ \txtd x=-\mu_n \bar{U}_2.
\end{aligned}
\end{equation}
Now, by multiplying the first two equations of~\eqref{eq:null} by $e_{\mu_n}$ and 
integrating over $\Omega$, using the boundary condition and~\eqref{UV}, we arrive at 
the following algebraic system for $\bar{U}_1$ and $\bar{U}_2$:
\begin{equation}
\label{eq:alg}
\begin{aligned}
  \bar{U}_1-g(u_1^*)\bar{U}_2 &=0, \\
	(\kappa\mu_n+\alpha)\bar{U}_2 -(f'(u_1^*)-\delta\mu_n)\bar{U}_1 &=0.
\end{aligned}
\end{equation}
If $\delta>f'(u_1^*)/\mu_n$ then the system~\eqref{eq:alg} has only the zero solution.
In this case, we would have $\mathcal{N}[\txtD_u\mathcal{F}(u^*,\delta)]=0$ for all $\delta$.
In order to have existence of a non-homogeneous solution we necessarily require 
$\delta\leq f'(u_1^*)/\mu_n$. In this case the system~\eqref{eq:alg} has a non-zero 
solution if and only if
\begin{equation}
\label{deltab}
\delta=:\delta^n_{\textnormal{b}}=-\frac{\kappa}{g(u_1^*)}+\frac{1}{\mu_n}
\Bigl[f'(u_1^*)-\frac{\alpha}{g(u_1^*)}\Bigr]=
\delta_{\textnormal{d}}+\frac{1}{\mu_n}\Bigl[f'(u_1^*)-\frac{\alpha}{g(u_1^*)}\Bigr].
\end{equation}
Taking $\delta=\delta^n_{\textnormal{b}}$, we can rewrite the first two equations 
of~\eqref{eq:null} as the system:
\begin{equation}
\label{A}
	\begin{pmatrix}
	\Delta U_1\\
	\Delta U_2
	\end{pmatrix}
	=\frac{1}{\kappa+\delta^n_{\textnormal{b}}g(u_1^*)}
	\begin{pmatrix}
	-g(u_1^*)f'(u_1^*) & g(u_1^*)\alpha \\
	-f'(u_1^*) & \alpha
	\end{pmatrix}
	\begin{pmatrix}
	U_1\\
	U_2
	\end{pmatrix}=:A
	\begin{pmatrix}
	U_1\\
	U_2
	\end{pmatrix}
\end{equation}
Using~\eqref{deltab} and computing the determinant and the trace of the 
matrix $A$ we find that its eigenvalues are $\lambda_1=0$ and $\lambda_2=-\mu_n$, 
where $\mu_n>0$ is a single eigenvalue of the problem~\eqref{eq:lapl}. Let $T$ 
be the matrix whose columns are the eigenvectors corresponding to $\lambda_1$ 
and $\lambda_2$ respectively:
\benn
T=
\begin{pmatrix}
\alpha & g(u_1^*) \\
f'(u_1^*) & 1
\end{pmatrix}.
\eenn
We have
\benn
T^{-1}AT=
\begin{pmatrix}
0 & 0 \\
0 & \mu_n
\end{pmatrix}.
\eenn
Then, by considering the transformation
\begin{equation}
\label{uvpq}
\begin{pmatrix}
	p\\
	q
	\end{pmatrix}=T^{-1}\begin{pmatrix}
	U_1\\
	U_2
	\end{pmatrix},
\end{equation}
it follows that the first two equations of~\eqref{eq:null} can be uncoupled 
and we find that
\begin{equation}
\label{eq:pq}
\begin{aligned}
  &\Delta p=0\qquad \text{~in }\Omega,\\
	&\Delta q=\mu_n q\qquad \text{in }\Omega,\\
	&\alpha\int_{\Omega}p(x)~\txtd x+ g(u_1^*)\int_{\Omega}q(x)~\txtd x=0,\\
	&\nabla p\cdot\nu=\nabla q\cdot\nu=0\qquad \text{on }\partial\Omega,
\end{aligned}
\end{equation}
where the genericity condition $-\kappa\neq \delta^n_ {\textnormal{b}} g(u_1^*)$ is used to 
obtain zero Neumann boundary conditions. Recall that $\mu_n$ is a simple eigenvalue of~\eqref{eq:lapl} 
with eigenfunction $e_{\mu_n}$. Observe that $\int_\Omega e_{\mu_n}(x)~ \txtd x=0$, which implies
that $p=0$ and $q=Ce_{\mu_n}$ for some constant $C$ are the solutions of \eqref{eq:pq}. Therefore,
it follows that
\begin{equation}
\label{span}
(U_1,U_2)^\top=C e_{\mu_n} (g(u_1^*),1)^\top .
\end{equation}
This shows that $\mathcal{N}[\txtD_u\mathcal{F}(u^*,\delta^n_ {\textnormal{b}})]= 
\textnormal{span}[e_{\mu_n} (g(u_1^*),1)^\top ]=:\text{span}[e^n_{\textnormal{b}}]$. In 
particular, the nullspace is one-dimensional and the result (L5) follows. 

To prove (L6), we argue by contradiction and suppose that~\eqref{eq:cross_cond} is not 
satisfied. Hence, by computing $\txtD_{\delta u}\cF(u^*,\delta^n_{\textnormal{b}})$, it 
follows there exists $(p,q)$ such that
\begin{equation}
\label{pq}
\begin{aligned}
  \Delta p-g(u_1^*)\Delta q =\mu_n g(u_1^*) e_{\mu_n}\quad & \text{in }\Omega,\\
	\kappa\Delta q +\delta^n_{\textnormal{b}}\Delta p -\alpha q+f'(u_1^*)p=0
	\quad &\text{in }\Omega,\\
	\int_{\Omega}p(x)~\txtd x=0,&\\
	\nabla p\cdot\nu = \nabla q\cdot\nu = 0\quad &\text{on }\partial\Omega.
\end{aligned}
\end{equation}
As in the first part of the proof, it is helpful to consider a suitable projection and 
we define $P$ and $Q$ as
\benn
P:=\int_\Omega p(x)e^n_{\textnormal{b}}(x)~\txtd x,\qquad 
Q:=\int_\Omega q(x)e^n_{\textnormal{b}}(x)~\txtd x.
\eenn
Multiplying the first two equations~\eqref{pq} by $e^n_{\textnormal{b}}$ and 
integrating over $\Omega$ and using boundary conditions one obtains an algebraic 
system for $P$ and $Q$ given by
\begin{equation}
\label{PQ}
	\begin{split}
	\left\{
	\begin{array}{rcl}
	\displaystyle
	P-g(u_1^*)Q &=&-g(u_1^*),\\
	(f'(u_1^*)-\delta^n_{\textnormal{b}}\mu_n)P-(\kappa\mu_n+\alpha) Q&=&0.
	\end{array}
	\right.
	\end{split}
\end{equation}
By the definition of $\delta^n_{\textnormal{b}}$, the determinant of the matrix of coefficients on the 
left-hand side of the system~\eqref{PQ} is zero. This implies that the inhomogeneous 
linear system has no solution. Hence the system~\eqref{pq} has no solutions and the 
result~\eqref{eq:cross_cond} in (L6) follows.
\end{proof}

Note that (L5)-(L6) are just the results (R2)-(R3) claimed in Theorem~\ref{thm:main_ana_bif}. 
By applying~\cite[Thm.~4.3]{ShiWang} we obtain the existence of a 
non-trivial branch of solutions.
Therefore, the local dynamics of the problem already shows that the entropy method 
cannot provide exponential decay to a distinguished steady state for all parameter
values. 

\section{Numerical Bifurcation Analysis -- Continuation Results}\label{sec.numerics.res}

In Section~\ref{sec.results1} we proved the existence of a weak solution for 
$\delta>\delta^*=-\kappa/\delta$ as well as global convergence to a steady state for 
$\delta>\delta_{\txte}$ ($\delta\neq 0$); in addition, $\delta_{\txte}$ converges to
$\delta^*=-\kappa/\gamma$ as $\alpha\ra +\I$ and $\delta_{\txte}$ converges to
$+\I$ as $\alpha\ra 0$.
In Section~\ref{sec.results2} we showed the existence of non-trivial solutions for 
$\delta=\delta^n_ {\textnormal{b}}$ where $\delta^n_ {\textnormal{b}}$ is defined in~\eqref{deltab}
and in particular $\delta^n_ {\textnormal{b}}$ could be bigger or smaller than $\delta_\txtd=\kappa/g(u_1^*)$
depending on $\alpha$.

The numerical continuation results presented in this section aim to augment and extend these 
results. To simplify the comparison to numerical results, we focus on the case
\benn
\kappa=1,\qquad g(s)=s(1-s),\qquad f(s)=s(1-s),
\eenn
which yields the condition $\delta>\delta^*=-4$ for the validity of the entropy method for $\alpha\ra +\I$.
As already mentioned, the values for $\delta^n_ {\textnormal{b}}$ depend on $\alpha$, so we are going
to study a case with $\alpha$ sufficiently large (Section~\ref{sec.case1}) and the case with $\alpha$ 
sufficiently small (Section~\ref{sec.case2}). Below we are going to define the meaning of sufficiently 
large and sufficiently small. First, we want to compare the values that we obtain for 
$\delta^n_ {\textnormal{b}}$ with the numerical results. The analytical problem did not include the small
parameter $\rho$ and the introduction of this term has the effect of shifting the bifurcation points.

\subsection{Comparison between the values of $\delta^n_ {\textnormal{b}}$}
\label{comparison_deltab}

The formula for $\delta^n_ {\textnormal{b}}$ given in the equation~\eqref{deltab} does not
consider the additional term $\rho$. Introducing this term, we get a new formula which reads
\be
\label{new_deltab}
 \delta^n_ {\textnormal{b}}=\frac{f'(u_1^*)}{\mu}-
\frac{(\kappa\mu+\alpha)(\mu+\rho)}{g(u_1^*)\mu^2}
 = \delta_{\textnormal{d}}+\frac{1}{\mu}\Bigl[ f'(u_1^*) - 
\frac{\kappa\rho+\alpha}{g(u_1^*)} -\frac{\alpha\rho}{g(u_1^*)\mu}\Bigr].
\ee
We observe that the formulas~\eqref{deltab} and~\eqref{new_deltab}, due to the presence of the term 
$\rho$, will not give the same values $\delta^n_ {\textnormal{b}}$ but the two equations correspond 
if we take $\rho=0$. We fix the following parameter values
\benn
(\kappa,\alpha,l,\bar{u}_1,\rho) =(1, 0.2, 20,0.594, 0.05).
\eenn
We are interested in computing the values of $\delta^n_ {\textnormal{b}}$ and to observe how the parameter
$\rho$ shifts the bifurcation branches.

\begin{table}
\begin{center}
\begin{tabular}{c||c|c|c|c|c|c|c|c|c}
 n & 1 & 2 & 3 & 4 & 5 & 6 & 7 & 8 & 9  \\
 \hline
\eqref{deltab} & -45.38 & -14.45 & -8.73 & -6.72 & -5.80 & -5.29 & -4.99 & -4.79 & -4.66  \\
\eqref{new_deltab} & -121.89 & -20.81 & -10.50 & -7.51 & -6.24 & -5.58  & -5.19 & -4.94 & -4.77 \\
\texttt{AUTO} & -121.89 & -20.81 & -10.50 & -7.51 & -6.24 & -5.58  & -5.19 & -4.94 & -4.77\\
\end{tabular}
\caption{\label{table:deltab}Comparison between the analytical and numerical bifurcation values.
The last two rows compare the numerical and analytical solutions with $0<\rho\ll1$.}
\end{center}
\end{table}

In Table~\ref{table:deltab} we reported the bifurcation points $\delta^n_ {\textnormal{b}}$ 
for $n\in\{1,2,\ldots,9\}$ computed with the two formulas~\eqref{deltab} and~\eqref{new_deltab} 
in comparison to 
the numerical continuation results using \texttt{AUTO}. The values detected using \texttt{AUTO} 
precisely correspond to the values computed with the formula~\eqref{new_deltab} as expected 
while the points are shifted in comparison to the values for $\rho=0$.

\subsection{Case 1: $\alpha$ sufficiently large}
\label{sec.case1}

Recall the formula for $\delta^n_ {\textnormal{b}}$ given in~\eqref{deltab}:
\[
\delta^n_{\textnormal{b}}=
\delta_{\textnormal{d}}+\frac{1}{\mu_n}\Bigl[f'(u_1^*)-\frac{\alpha}{g(u_1^*)}\Bigr].
\]
We observe that if $\alpha>f'(u_1^*)g(u_1^*)$ then $\delta^n_ {\textnormal{b}}<\delta_\txtd$
and the branches will approach the limit value $\delta_\txtd$ for $n\rightarrow\infty$. Since 
we are using~\eqref{new_deltab}, the condition on $\alpha$ is
\benn
\alpha>\mu_n\Bigl[ \frac{f'(u_1^*)g(u_1^*) -\kappa\rho}{\rho + \mu_n}\Bigr]
\eenn
and, in the case of an interval we can compute the eigenvalues $\mu$.
So, $\alpha$ sufficiently large means
\be
\label{eq:condalpharho}
\alpha>\Bigl(\frac{n\pi}{l}\Bigr)^2\Bigl[ \frac{f'(u_1^*)g(u_1^*) -
\kappa\rho}{\rho + (\frac{n\pi}{l})^2}\Bigr].
\ee
Figure~\ref{fig:01} shows a continuation calculation for fixed parameters 
\benn
(\kappa,\alpha,l,\bar{u}_1,\rho)=(p_2,p_3,p_4,p_5,p_6)=(1,0.2,12,0.594,0.05)
\eenn
using $\delta$ as the primary bifurcation parameter. We observe that the condition 
on $\alpha$ is satisfied since the right-hand side of~\eqref{eq:condalpharho} is 
negative for all $n\in\N$ and $\alpha=0.2$. The steady state we start the continuation 
with is given by
\benn
(u_1^*,u_2^*)=(\bar{u}_1,f(\bar{u}_1)/\alpha).
\eenn

\begin{figure}[htbp]
\psfrag{u}{\small{$u_1$}}
\psfrag{x}{\small{$x$}}
\psfrag{a}{\small{(a)}}
\psfrag{b1}{\small{(b1)}}
\psfrag{b2}{\small{(b2)}}
\psfrag{b3}{\small{(b3)}}
\psfrag{b4}{\small{(b4)}}
\psfrag{delta}{\small{$\delta$}}
\psfrag{L2}{\scriptsize{$\|z\|_{L^2}$}}
	\centering
		\includegraphics[width=1\textwidth]{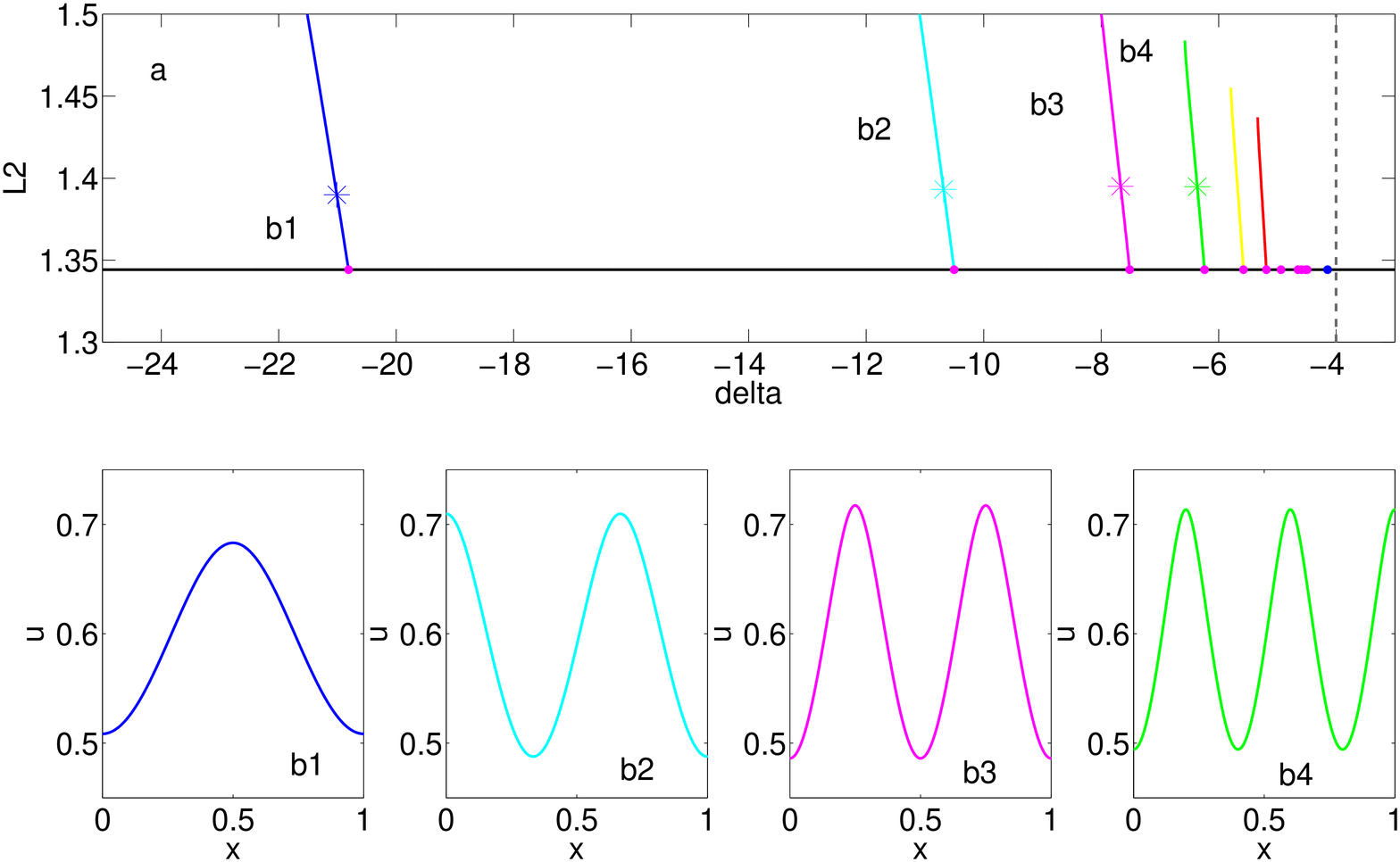}
		\caption{\label{fig:01}Continuation calculation for the system \eqref{eq:BVP_AUTO} with parameter
		values $(\kappa,\alpha,l,\bar{u}_1,\rho)=(p_2,p_3,p_4,p_5,p_6)=(1,0.2,20,0.594,0.05)$ and 
		primary bifurcation parameter $\delta$. (a) Bifurcation diagram in $(\delta,
		\|z\|_{L^2})$-space showing the parameter on the horizontal axis and the solution norm
		on the vertical axis. Some of the detected bifurcation points are marked as circles (magenta). 
		The last branch point (blue circle) is not a true bifurcation point but results from
		the degeneracy $\delta=-\kappa/g(u_1^*)=:\delta_\txtd$. At the other branches points 
		(magenta, filled circles) non-homogeneous solution branches (blue, cyan, magenta, green...) 
		bifurcate via single eigenvalue crossing. The value $\delta^*=-\kappa/\gamma=-4$
		is marked by a vertical grey dashed line. (b) Solutions are plotted for $(x,u_1=u_1(x))$ at 
		certain points on the non-homogeneous branches; the solutions are marked in (a) using crosses.
		}
\end{figure} 

We begin the continuation at $\delta=-25$ and we find only one bifurcation point when $\delta$ is decreasing, 
i.e.\ for $\delta<-25$. This result is expected since $\delta^1_ {\textnormal{b}}=-121.889$ is the value corresponding to the first eigenvalue. Moreover, we do not detect any bifurcations for $\delta>-4=\delta^*$.
The interesting results in the bifurcation calculation in Figure~\ref{fig:01} occur when we increase
the primary bifurcation parameter $\delta$. In this case, several branch points are detected, in particular
the closer we are to the value $\delta_\txtd$, the more bifurcation points are found. 
In Figure~\ref{fig:01}, we have shown the first six branch points detected obtained upon increasing $\delta$. 
The point detected at $\delta = -20.8116$ corresponds to the second non-trivial bifurcation branch. 
There are more and more points as we get closer to $\delta_\txtd$.
The last point detected (in blue) is not a bifurcation point but corresponds to the degeneracy at
\benn
\kappa/g(u_1^*)= -1/(0.594(1-0.594))\approx-4.1466.
\eenn
The remaining detected branch points in Figure~\ref{fig:01} are true bifurcation points. This 
numerical result is in accordance with the analytical results on the existence of bifurcations in 
Theorem~\ref{thm:main_ana_bif}. In fact, one can carry out the same calculation as in 
Section~\ref{sec.bifurcation.res}. At each bifurcation point, a simple eigenvalue
crosses the imaginary axis. One can use the branch switching algorithm implemented in \texttt{AUTO} to
compute the non-homogeneous families of solutions as shown for four points in
Figure~\ref{fig:01}(a). In Figure~\ref{fig:01}(b), we show a representative solution $u_1=u_1(x)$ on each 
of the four solution families. The solutions are non-homogeneous steady states and have interface-like
behaviour in the spatial variable. Each family has a characteristic number of these interfaces.
There are families with even more interfaces than the one shown in Figure~\ref{fig:01}(b4), which
can be found upon increasing $\delta$ even further; we are not interested in these highly oscillatory 
solutions here.\medskip

\begin{figure}[htbp]
\psfrag{u}{\small{$u_1$}}
\psfrag{x}{\small{$x$}}
\psfrag{a}{\small{(a)}}
\psfrag{b}{\small{(b)}}
\psfrag{delta}{\small{$\delta$}}
\psfrag{L2}{\scriptsize{$\|z\|_{L^2}$}}
	\centering
		\includegraphics[width=0.8\textwidth]{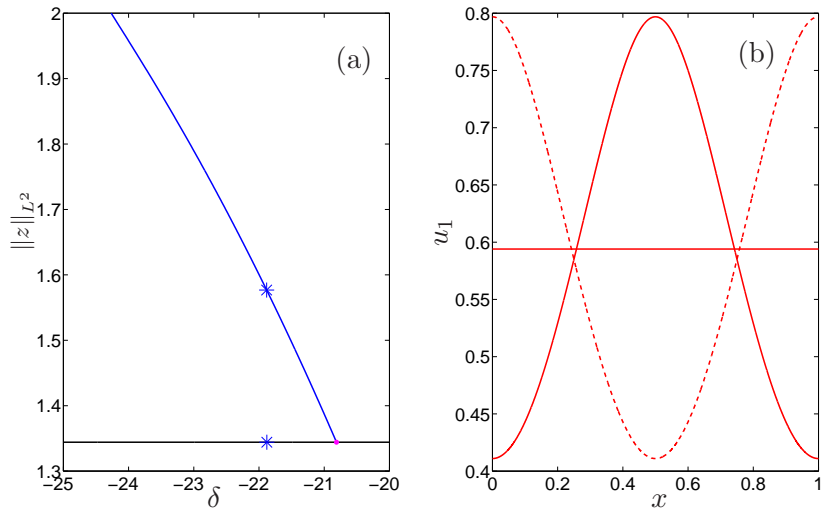}
		\caption{\label{fig:02}Continuation calculation for the system \eqref{eq:BVP_AUTO} as in 
		Figure~\ref{fig:01} 
		with a focus on the second bifurcation point (filled circle, magenta). One can show that by
		using two different local branching directions that two different non-homogeneous solution 
		branches (red) bifurcate 
		via single eigenvalue crossing but the two branches contain solutions with identical 
		$L^2$-norm for the same parameter value. This is a result of a symmetry in the problem. 
		(b) Three different solutions plotted in $(x,u_1=u_1(x))$-space 
		at the parameter value $\delta=-21.8819$. The three solutions are marked in (a) using crosses.
		}
\end{figure} 

Another observation regarding the continuation run in Figure~\ref{fig:01} is reported in more detail
in Figure~\ref{fig:02} with a focus on the second bifurcation point. It is shown that there are actually 
two different branches bifurcating at the same point with families of non-homogeneous solutions that
are symmetric. In particular, one non-trivial solution branch can be transformed into the other by 
considering $u\mapsto 1-u$; as an illustration we refer to three representative numerical solutions on the 
three branches originating at the second bifurcation point as shown in Figure~\ref{fig:02}(b).

\subsection{Case 2: $\alpha$ sufficiently small}
\label{sec.case2}

As specified in (M7) when $\alpha<f'(u_1^*)g(u_1^*)$ then $\delta^n_ {\textnormal{b}}>\delta_\txtd$ and 
this means that the branches will approach the limit value $\delta_\txtd$ from the right.
As pointed out in~\ref{comparison_deltab}, the condition on $\alpha$ is more complicated since our model 
contains $\rho$. The condition on $\alpha$ becomes
\benn
0<\alpha<\mu_n\Bigl[\frac{f'(u_1^*)g(u_1^*) - \kappa\rho}{\rho + \mu_n}\Bigr],
\eenn
i.e. we must choose an $\alpha$ which satisfied the inequality for each single $\mu_n$. We 
fix
\benn
(\kappa,\alpha,l,\bar{u}_1,\rho) =(1, 0.001, 50, 0.211325, 0.05)
\eenn
for the numerical continuation in this section. 

\begin{figure}[htbp]
\psfrag{u}{\small{$u_1$}}
\psfrag{x}{\small{$x$}}
\psfrag{a}{\small{(a)}}
\psfrag{b1}{\small{(b1)}}
\psfrag{b2}{\small{(b2)}}
\psfrag{b3}{\small{(b3)}}
\psfrag{delta}{\small{$\delta$}}
\psfrag{L2}{\scriptsize{$\|z\|_{L^2}$}}
	\centering
		\includegraphics[width=1\textwidth]{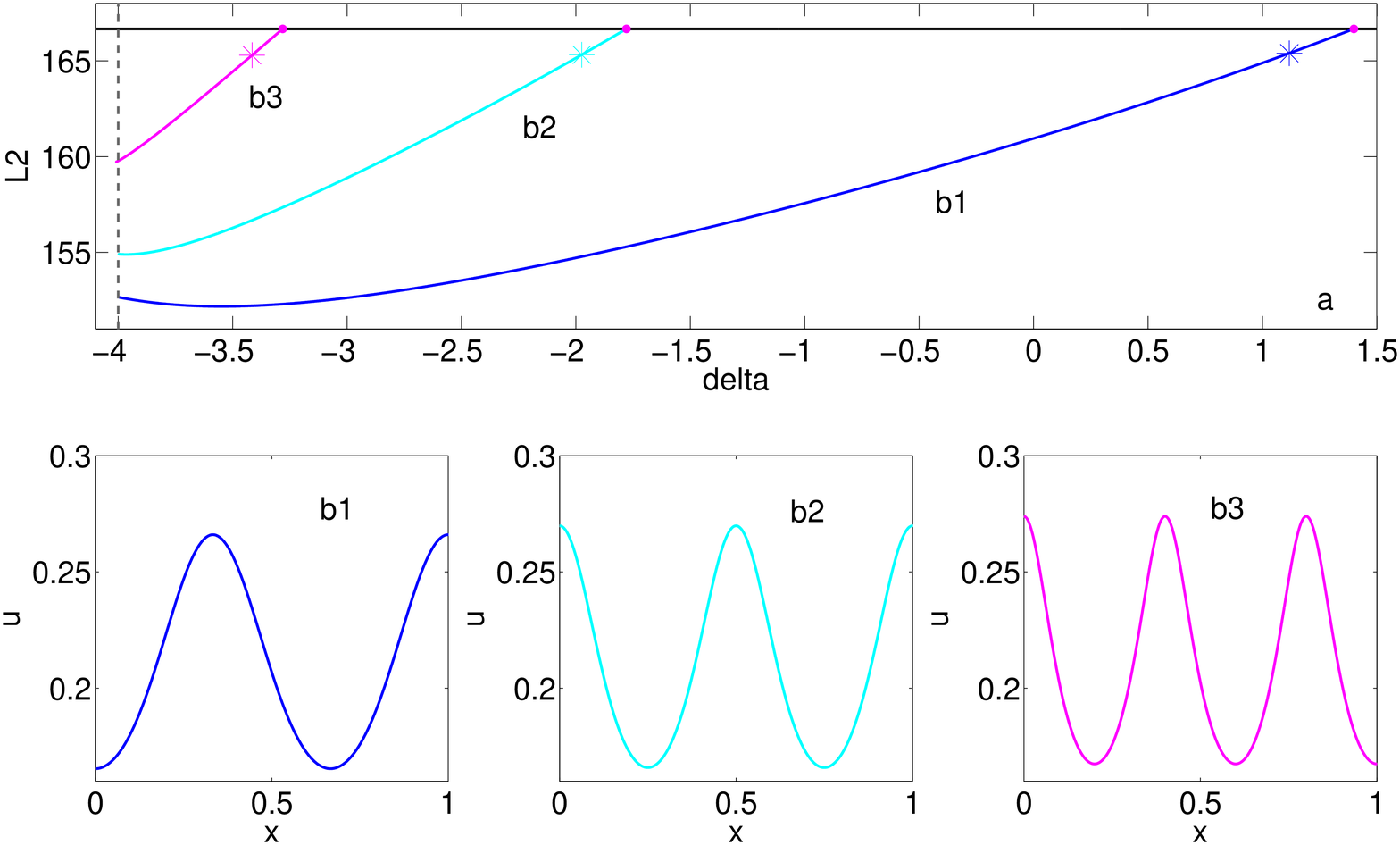}
		\caption{\label{fig:bif_diag1}Continuation calculation for the system \eqref{eq:BVP_AUTO} with 
		parameter values $(\kappa,\alpha,l,\bar{u}_1,\rho)=(p_2,p_3,p_4,p_5,p_6)=(1,0.0001,50,0.211325,0.05)$ 
		and primary bifurcation parameter $\delta$. (a) Bifurcation diagram in $(\delta,
		\|z\|_{L^2})$-space showing the parameter on the horizontal axis and the solution norm
		on the vertical axis. The detected bifurcation points are marked as circles (magenta). 
		At the three branch points (magenta, filled circles) non-homogeneous solution branches 
		(blue, cyan, magenta) corresponding to $\delta^3_ {\textnormal{b}}, \delta^4_ {\textnormal{b}},
		\delta^5_ {\textnormal{b}}$ bifurcate via single 
		eigenvalue crossing. The value $\delta^*=-\kappa/\gamma=-4$
		is marked by a vertical grey dashed line. (b) Solutions are plotted for $(x,u_1=u_1(x))$ at certain
		points on the non-homogeneous branches; the solutions are marked in (a) using crosses.}
\end{figure} 

With these values the condition on $\alpha$ is 
given by $0<\alpha<0.0033827$ which is satisfied. We also find that with our choices 
\[
\delta_\txtd<\delta^n_ {\textnormal{b}}<\delta^*<\delta^5_ {\textnormal{b}}<
\delta^4_ {\textnormal{b}}<\delta^3_ {\textnormal{b}}<\delta^2_ {\textnormal{b}}<
\delta^1_ {\textnormal{b}}<\delta_\txte,\quad n\geq6,
\]
i.e. there are some bifurcation points which are bigger than $\delta^*$ and some which are smaller but all
of them are bigger than $\delta_\txtd$. We begin the continuation at $\delta=3$ and we detect only two 
more branches when we increase $\delta$ at $\delta=43.4851$ and $\delta=9.98041$ which correspond to 
$\delta^1_ {\textnormal{b}}$ and $\delta^2_ {\textnormal{b}}$. We focus on the branches for $n\in\{1,2,3,4,5\}$ 
such that $\delta^n_ {\textnormal{b}}>\delta^*$. This case is represented in Figure~\ref{fig:bif_diag1}.

\begin{figure}[htbp]
\psfrag{u}{\small{$u_1$}}
\psfrag{x}{\small{$x$}}
\psfrag{a}{\small{(a)}}
\psfrag{b1}{\small{(b1)}}
\psfrag{b2}{\small{(b2)}}
\psfrag{b3}{\small{(b3)}}
\psfrag{delta}{\small{$\delta$}}
\psfrag{L2}{\scriptsize{$\|z\|_{L^2}$}}
	\centering
		\includegraphics[width=1\textwidth]{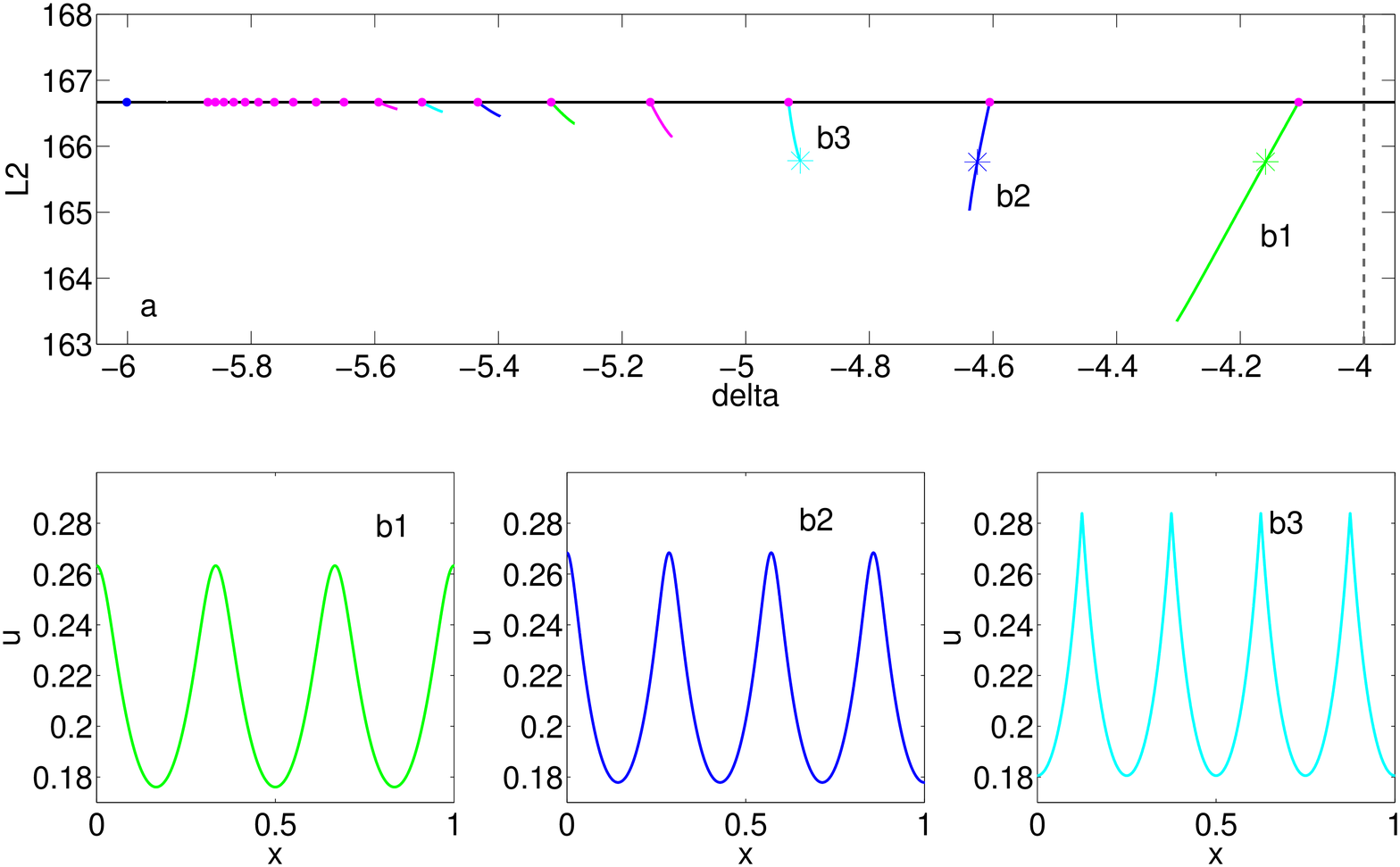}
		\caption{\label{fig:bif_diag2}Continuation calculation for the system \eqref{eq:BVP_AUTO} with parameter
		values $(\kappa,\alpha,l,\bar{u}_1,\rho)=(p_2,p_3,p_4,p_5,p_6)=(1,0.0001,50,0.211325,0.05)$ and 
		primary bifurcation parameter $\delta$. (a) Bifurcation diagram in $(\delta,
		\|z\|_{L^2})$-space showing the parameter on the horizontal axis and the solution norm
		on the vertical axis. Some of the detected bifurcation points are marked as circles (magenta). 
		The last branch point (blue circle) is not a true bifurcation point but results from
		the degeneracy $\delta=-\kappa/g(u_1^*)=:\delta_\txtd$. At the other branch points 
		(magenta, filled circles) non-homogeneous solution branches (green, blue, cyan) bifurcate via single 
		eigenvalue crossing. The value $\delta^*=-\kappa/\gamma=-4$
		is marked by a vertical grey dashed line. (b) Solutions are plotted for $(x,u_1=u_1(x))$ at certain
		points on the non-homogeneous branches; the solutions are marked in (a) using crosses.}
\end{figure} 

Numerically we observe that all the branches stop when they reach the critical value $\delta^*$. 
Next, we consider $n\geq 6$ such that $\delta_\txtd<\delta^n_ {\textnormal{b}}<\delta^*$ as 
reported in Figure~\ref{fig:bif_diag2}. In this case there are two critical values: $\delta^*=-4$ 
(dashed line) and $\delta_\txtd= -6$ (blue circle). The branches detected for a $\delta$ close to 
$\delta^*$ have the same direction as the branches detected for $\delta>\delta^*$; but
starting from a certain $n$, in this case $n=8$, we notice that the branches change the direction. Probably
this behaviour is due to the fact that the branches cannot cross the value $\delta=\delta_\txtd$.
We do no detect any branch for $\delta<\delta_\txtd$.

In the range between $\delta_\txtd$ and $\delta^*$ the branches do not seem to overlap. Numerically, 
one observes that the branches get shorter and shorter due to the numerical continuation breaking down as 
the branches approach $\delta_\txtd$. Looking at the shape of the solutions in the different branches we 
can observe that they have more and more interfaces as we approach the limiting value $\delta_\txtd$. 
Moreover, the solutions inside a fixed branch get sharper and sharper peaks along the branch
(see for example the cyan branch).

\subsection{Continuation in $\rho$}
\label{rho}

The next question is if we can find non-homogeneous steady states also for the 
original problem with $\rho=0$. This can be achieved by using a homotopy-continuation idea. 

\begin{figure}[htbp]
\psfrag{u}{\small{$u_1$}}
\psfrag{x}{\small{$x$}}
\psfrag{b1}{\small{(b1)}}
\psfrag{b2}{\small{(b2)}}
\psfrag{b3}{\small{(b3)}}
\psfrag{a1}{\small{(a1)}}
\psfrag{a2}{\small{(a2)}}
\psfrag{a3}{\small{(a3)}}
\psfrag{c1}{\small{(c1)}}
\psfrag{c2}{\small{(c2)}}
\psfrag{c3}{\small{(c3)}}
\psfrag{rho}{\small{$\rho$}}
\psfrag{L2}{\scriptsize{$\|z\|_{L^2}$}}
	\centering
		\includegraphics[width=1\textwidth]{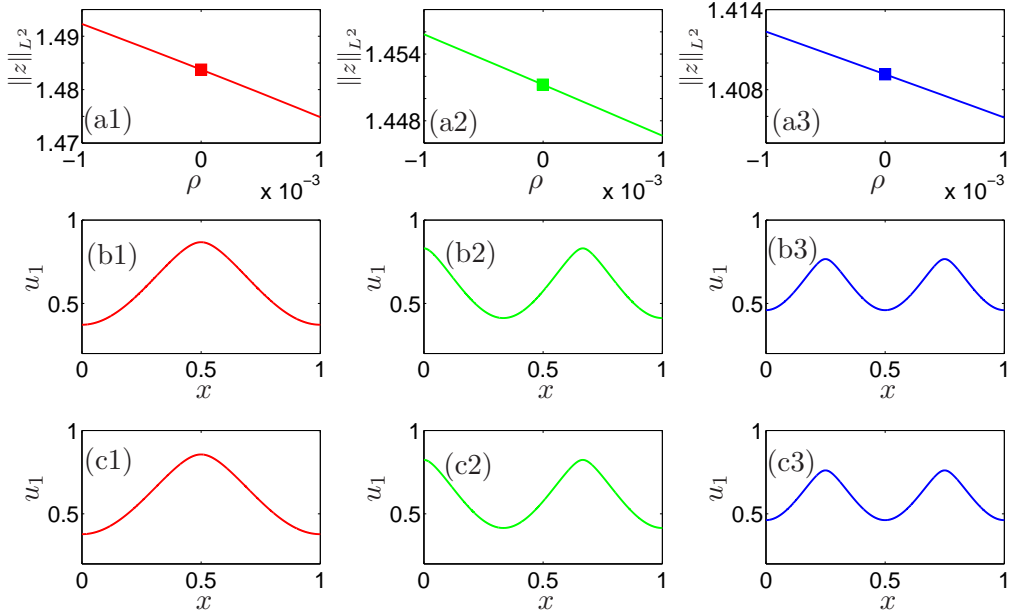}
		\caption{\label{fig:04}Continuation calculation for the system \eqref{eq:BVP_AUTO} starting with the 
		same basic parameter values as in Figure~\ref{fig:01} but with $\rho=0.001$.
		We stop the continuation at the 
		solution points for a certain $\delta$ (as done in Figure~\ref{fig:01}(a)) and change 
		from $\delta$ as a primary continuation parameter to $\rho$ as a primary parameter with 
		the goal to decrease the parameter to $\rho=0$. The values for $\delta$ are
		$\delta=-16$ for the red branch, $\delta=-9.4$ for the green branch and $\delta=-7$ for the blue one.
		(a1)-(a3) Bifurcation diagrams in 
		$(\rho,\|z\|_{L^2})$-space. The starting point for the continuation is at the right boundary
		where $\rho=0.001$ and then $\rho$ is decreased. (b1)-(b3) Solutions obtained on the bifurcation
		branches above at the point $\rho=0$ (points are marked with squares in (a1)-(a3)).
		(c1)-(c3) Solutions obtained on the bifurcation
		branches for the initial system with $\rho=0.001$. We can observe that also for $\rho=0$ the 
		solutions have a non-trivial herding-type profile.
		}
\end{figure} 

First, 
we continue the problem in $\delta$ and compute the non-homogeneous solution branches. Then we pick
a steady state on the non-homogeneous branch and switch to continuation in $\rho$ while keeping
$\delta$ fixed. The results of this strategy are shown in Figure~\ref{fig:04} (for $\alpha=0.2$)
and in Figure~\ref{fig:05} (for $\alpha=0.001$). For the first three
solutions shown in Figure~\ref{fig:01}(b), this strategy works if we start from a very small $\rho$.
Figure~\ref{fig:04}(c) shows the solution in the branch for a $\rho\neq 0$: we notice that
the solutions for the case $\rho=0$ keep the non-constant profile as for $\rho\neq0$ yielding
relevant herding solutions for applications.

\begin{figure}[htbp]
\psfrag{u}{\small{$u_1$}}
\psfrag{x}{\small{$x$}}
\psfrag{a}{\small{(a)}}
\psfrag{b}{\small{(b)}}
\psfrag{rho}{\small{$\rho$}}
\psfrag{L2}{\scriptsize{$\|z\|_{L^2}$}}
	\centering
		\includegraphics[width=0.8\textwidth]{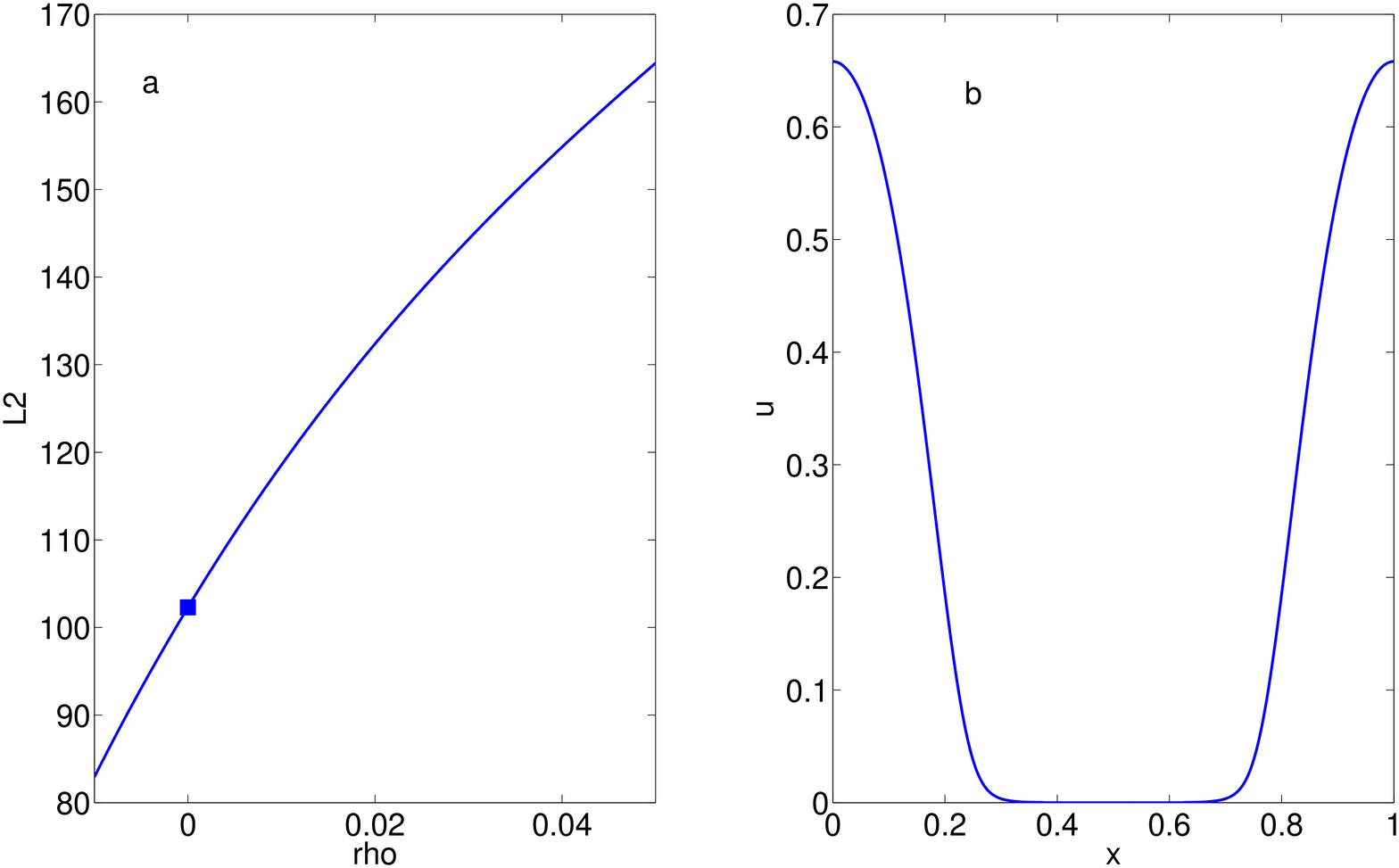}
		\caption{\label{fig:05}Continuation calculation for the system \eqref{eq:BVP_AUTO} starting
		with the same parameter value and as in Figure~\ref{fig:bif_diag1}.
		We stop the continuation at $\delta=-9$ (as done in Figure~\ref{fig:bif_diag1}(a)) and change 
		from $\delta$ as a primary continuation parameter to $\rho$ as a primary parameter with 
		the goal to decrease the parameter to $\rho=0$. (a) Bifurcation diagram in 
		$(\rho,\|z\|_{L^2})$-space. The starting point for the continuation is at the right boundary
		where $\rho=0.05$ and then $\rho$ is decreased. (b) Solution on the second branch $\delta^2_ {\textnormal{b}}$
		of non-homogeneous steady states at $\rho=0$ (point is marked with squares in (a)).
		}
\end{figure} 

In the case with $\alpha$ sufficiently small, the strategy works better and
we indeed find non-homogeneous steady states for $\rho=0$ as shown in Figure~\ref{fig:05}(b).
Moreover we can also obtain herding solutions. We use the starting parameter values
\benn
(\kappa,\alpha,l,\bar{u}_1,\rho)=(1,0.001,50,0.211325,0.05).
\eenn
We start from $\delta=10$ and the first branch we detect is $\delta^2_ {\textnormal{b}} = 9.98041$.
Once we are in this branch, we continue in $\rho$ for a fixed $\delta$ (in this case
$\delta=-9$). For information herding models, solutions which are of particular importance are 
those with sharp interfaces between the endstates, i.e., the solution is near zero and near one 
in certain regions with sharp interfaces in between. 
These solutions represent a herding effect in the sense of sharply split opinions.
More precisely, they indicate for which values of the information variable
$x$ we observe a herding behaviour, i.e.\ a concentration of individuals
($u\approx 1$) at certain values of $x$. Figure \ref{fig:05}(b) shows herding
in the interval $[0,0.2]\cup[0.8,1]$, while only a few individuals
adopt the information value in $[0.3,0.7]$.

\subsection{Solutions and other parameters}

In this section we focus on the case with $\alpha$ sufficiently small. We are interested in studying, 
how the solutions change depending on the other parameters $\kappa$ and $l$. We fix as starting parameters 
\benn
(\kappa,\alpha,l,\bar{u}_1,\rho)=(1,0.001,50,0.211325,0.05)
\eenn
and consider the branch $\delta^2_ {\textnormal{b}}$. We study the solutions depending on the different 
parameters. In Figure~\ref{fig:10} we show changes along the branch (which bifurcates at $\delta=9.98041$).
We observe that the shape is the same along the branch but the interfaces sharpen as $\delta$
is decreased.\medskip

\begin{figure}[htbp]
\psfrag{u}{\small{$u_1$}}
\psfrag{x}{\small{$x$}}
\psfrag{a}{\small{(a)}}
\psfrag{b}{\small{(b)}}
\psfrag{c}{\small{(c)}}
	\centering
		\includegraphics[trim = 00mm 70mm 00mm 65mm, clip,width=1\textwidth]{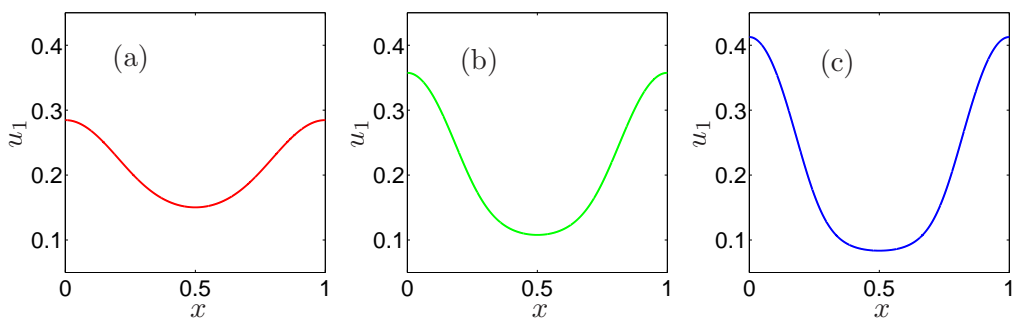}
		\caption{\label{fig:10} Solutions along the branch $\delta^2_ {\textnormal{b}}$ for the 
		system~\eqref{eq:BVP_AUTO} with parameter values 
		$(\kappa,\alpha,l,\bar{u}_1,\rho)=(p_2,p_3,p_4,p_5,p_6)=(1,0.001,50,0.211325,0.05)$.
		(a) Solution of non-homogeneous steady states at $\delta=8.72901$. (b) Solution 
		of non-homogeneous steady states at $\delta = 5.76477$. (c) Solution of non-homogeneous steady 
		states at $\delta=1.548$.}
\end{figure} 

In Figure~\ref{fig:11} we show how the solution changes with the length of the domain.
We consider $l=20$, $l=50$ and $l=100$. The branch $\delta^2_ {\textnormal{b}}$
is detected at $\delta=-3.28144, 9.98041, 43.4851$ respectively.
Since we consider the same branch, the shape does not change and length of the domain 
shifts the bifurcation points and just scales the solution.

\begin{figure}[htbp]
\psfrag{u}{\small{$u_1$}}
\psfrag{x}{\small{$x$}}
\psfrag{a}{\small{(a)}}
\psfrag{b}{\small{(b)}}
\psfrag{c}{\small{(c)}}
	\centering
		\includegraphics[trim = 00mm 70mm 00mm 65mm, clip,width=1\textwidth]{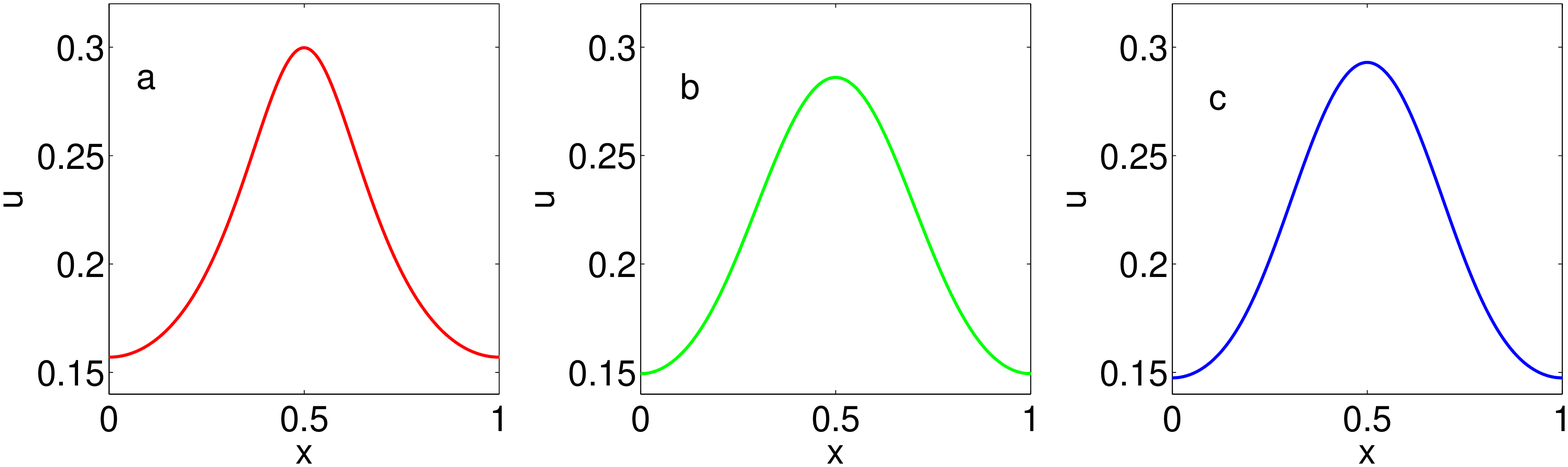}
		\caption{\label{fig:11} Solutions in the branch $\delta^2_ {\textnormal{b}}$ for the 
		system~\eqref{eq:BVP_AUTO} with parameter values $(\kappa,\alpha,\bar{u}_1,\rho)
		=(p_2,p_3,p_5,p_6)=(1,0.001,0.211325,0.05)$.
		(a) Solution of non-homogeneous steady states at $\delta=-3.5154$, $l=20$. (b) Solution 
		of non-homogeneous steady states at $\delta =8.93964 $, $l=50$. (c) Solution of non-homogeneous 
		steady states at $\delta=37.9117$, $l=100$.}
\end{figure}  

When we change the parameter $\kappa$ the bifurcation points are also simply shifted.
We consider $\kappa=1$, $\kappa=5$ and $\kappa=10$. The branch $\delta^2_ {\textnormal{b}}$
is detected at $\delta=9.98041,-92.2877,-214.999$ respectively. Moreover, for the first 
case the branches approach the value $\delta_\txtd$ from the right, while in the other two cases 
from the left. As for the previous case we consider three different solutions with (almost) the 
same norm ($163.863$ for the case (a), $163.872$ for (b) and $163.911$ for (c)).

\begin{figure}[htbp]
\psfrag{u}{\small{$u_1$}}
\psfrag{x}{\small{$x$}}
\psfrag{a}{\small{(a)}}
\psfrag{b}{\small{(b)}}
\psfrag{c}{\small{(c)}}
	\centering
		\includegraphics[trim = 00mm 70mm 00mm 65mm, clip,width=1\textwidth]{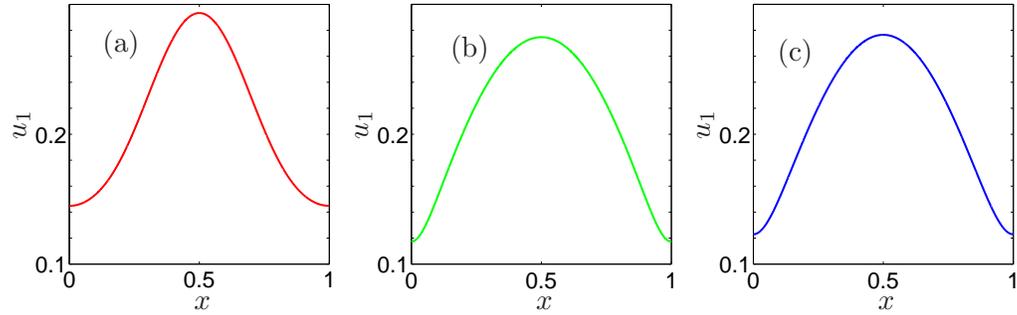}
		\caption{\label{fig:12} Solutions in the branch $\delta^2_ {\textnormal{b}}$ for the 
		system~\eqref{eq:BVP_AUTO} with parameter values $(\alpha,l,\bar{u}_1,\rho)=(p_3,p_4,p_5,p_6)
		=(0.001,50,0.211325,0.05)$.
		(a) Solution of non-homogeneous steady states at $\delta=8.72901$, $\kappa=1$. (b) Solution 
		of non-homogeneous steady states at $\delta =-92.2877 $, $\kappa=5$. (c) Solution of 
		non-homogeneous steady states at $\delta=-220.578$, $\kappa=10$.}
\end{figure} 

In summary, we conclude that $\kappa$ and $l$ do not seem to be the parameters of primary
importance in our context as we can re-obtain similar solutions and similar bifurcation structures 
for different values of $\kappa$ and $l$ upon varying $\delta,\alpha$ as primary parameters.

\section{Outlook}\label{sec.outlook}

So far, relatively little attention has been devoted to the study of the parameter space interfaces 
of different mathematical methods. In this contribution, we have analysed as an example a cross-diffusion
herding model to understand where, and how, the global nonlinear analysis approach via entropy variables 
is connected to bifurcation analysis techniques from dynamical systems. We have shown that both approaches
encounter similar problems regarding the degeneracy of the diffusion matrix and we were able to cover 
different parameter regimes by combining the results of the two methods.

This paper is only a first starting point. Here we shall just mention a few ideas for future work. 

The next step is to analyse the regime $\alpha\ra 0$ and to check whether the limitation 
in \eqref{1.mu} on $\alpha$ can be improved, or not. In this regard, one also has to consider in
which sense the forward problem should be interpreted for moderate and small values of $\alpha$ and for 
$\delta<\delta_{\textnormal{d}}$. Recent work~\cite{LionsWien} suggests that one 
should not only use the notion of Petrovskii ellipticity for the stationary problem~\cite{ShiWang} 
but also consider it in the parabolic context; see the classical survey~\cite{AgranovichVishik}.

The next step is to expand the approach to other examples. In particular, many reaction-diffusion 
systems as well as other classes of PDEs have natural entropies, which can be used to study global 
existence and convergence properties. In the nonlinear case, one frequently can also employ approaches 
from dynamical systems to understand the dynamics of the PDE. Using a similar approach as 
we presented here could be illuminating for other examples. For example, it is natural to conjecture 
that there are examples in applications, which exhibit the following characteristics:

\begin{itemize}
 \item[(Z1)] There exists one fixed parameter region in which the entropy method yields global 
 decay. Upon variation of a single parameter, the validity boundary of the entropy method coincides 
 precisely with an isolated local supercritical bifurcation point.
 \item[(Z2)] There exists one fixed parameter region in which the entropy method yields global 
 decay. Upon variation of a single parameter, the validity boundary of the entropy method does 
 not coincide with a local bifurcation point. Instead, the obstruction is a global bifurcation 
 branch in parameter space with a fold point precisely at the validity boundary.
\end{itemize}

In this work, we apparently found a more complicated case as shown in Figure~\ref{fig:00}. 
However, it seems plausible that the cases (Z1)-(Z2) should occur even in classical problems without 
cross-diffusion, i.e.\ reaction-diffusion equations with a diagonal positive-definite diffusion matrix. 
Determining whether this is true for several classical examples 
from applications is an interesting open problem.

Regarding the entropy method \cite{Carrilloetal,DesvillettesFellner1}, it would be interesting to 
investigate in more detail parametric scenarios for its validity regime. For example, the question 
arises whether it is possible to find criteria for the validity range that are computable for 
entire classes of PDEs. The entropy approach relies on upper bounds. Although the bounds we 
present here turn out to be sharp in the sense of global decay dynamics in a suitable singular limit,
this may not always be easy to achieve as demonstrated by the $\alpha\ra 0$ case discussed above.  
It would be relevant to estimate a priori, which regime in 
parameter space one fails to cover if certain non-optimal upper bounds are used. As above, carrying 
this out for several examples could already be very illuminating.

Regarding the analytical and numerical bifurcation analysis, there are multiple strategies to
deal with the problem of mass conservation, or more generally with higher-dimensional solution
manifolds. For example, one may try to compute the entire solution family of steady states 
parametrized by the mass numerically \cite{Henderson1,DankowiczSchilder}, which yields a 
numerical continuation problem for higher-dimensional manifolds and not only curves. 
Furthermore, we have focused on the numerical problem in the one-dimensional setup and 
computing the two- and three-space dimension cases could be interesting 
\cite{KuehnEllipticCont,UeckerWetzelRademacher}. Regarding analytical generalizations, 
a possible direction is to view $\delta^*$ as a singular limit 
and phrase the problem as a perturbation problem \cite{Ni,Fife,AchleitnerKuehn}. 


\bibliographystyle{alpha}
\bibliography{./JKT}

\newpage 

\tableofcontents

\end{document}